\documentclass[11pt]{amsart}
\input{amssym.def}
\input{amssym}
\newtheorem{pro}{Proposition}[section]
\newtheorem{lem}[pro]{Lemma}

\newtheorem{theo}[pro]{Theorem}
\newtheorem{defi}[pro]{Definition}
\newtheorem{cor}[pro]{Corollary}
\newtheorem{remk}[pro]{Remark}
\newtheorem{assu}[pro]{Assumption}

\newcommand{\ep}{\varepsilon}
\newcommand{\al}{\alpha}
\newcommand{\om}{\omega}
\newcommand{\vp}{\varphi}
\newcommand{\la}{\lambda}

\newcommand{\dt}{\partial_t}
\newcommand{\ds}{\partial_s}
\newcommand{\Rep}{\mbox{Re}\;}

\newcommand{\lra}{\longrightarrow}
\newcommand{\lmt}{\longmapsto}

\newcommand{\nrm}[1]{\mbox{ $ \displaystyle \left\| {#1} \right\| $} }
\newcommand{\nri}[1]{\mbox{ $ \nrm{ {#1} }_{\infty} $} }
\newcommand{\nrE}[1]{\mbox{ $ \nrm{ {#1} }_{E} $} }

\newcommand{\fk}[1]{ \left( {#1} \right) }

\newcommand{\dual}[1]{ \langle {#1} \rangle }
\newcommand{\semim}[1]{ [ {#1} ]_- }     %\semi = Brackets  [ , ]-
\newcommand{\bk}[1]{ \left\{ {#1} \right\} }
\newcommand{\btr}[1]{\mbox{ $ \left| {#1} \right| $ }}

\newcommand{\re}{{\bf\Bbb R}}
\newcommand{\rep}{{\bf\Bbb R^+}}
\newcommand{\za}{{\bf\Bbb N}}

\newcommand{\rat}{{\bf\Bbb Q}}

\newcommand{\jv}[2]{\mbox{$ \fk{{#1},\mbox{$ {#2} $}} $} }
\newcommand{\rv}[1]{ \jv{\re}{#1} }
\newcommand{\rpx}{ \jv{\rep}{X} }

\newcommand{\rx}{\rv{X}}

\newcommand{\apx}{\mbox{ $ AP \rx $ } }

\newcommand{\ctq}{\mbox{$ C([0,T]^2) $}}

\newcommand{\bucrpx}{\mbox{$BUC \rpx $ }}
\newcommand{\bucrx}{\mbox{$BUC \rx $ }}

\newcommand{\intO}{\int_{0}^{\infty}}

\newcommand{\res}[1]{\frac{1}{\lambda}e^{-\frac{#1}{\lambda}}}
\newcommand{\resm}[1]{\frac{1}{\mu}e^{-\frac{#1}{\mu}}}
\newcommand{\nres}[1]{e^{-\frac{#1}{\lambda}}}
\newcommand{\nresl}[1]{\frac{e^{-\frac{#1}{\lambda}}}{\la}}
\newcommand{\nresm}[1]{e^{-\frac{#1}{\mu}}}

\newcommand{\Funk}[5]{ \begin{array}{ccccc}
                       {#1} & : & {#2} & \lra & {#3} \\
                            &   & {#4} & \lmt & \displaystyle{#5} 
                       \end{array}                       }

\newcommand{\Jl}{J_{\lambda}}
\newcommand{\Jlw}{J_{\lambda}^{\om}}

\newcommand{\hw}{h^{\om}}
\newcommand{\Lw}{L^{\om}}

\newcommand{\dl}{\frac{1}{\lambda}}

\newcommand{\ki}{\omega}

\newcommand{\dm}{\frac{1}{\mu}}
\newcommand{\lpmplmw}{\lambda+\mu-\lambda\mu\ki}
\newcommand{\alm}{\frac{\mu}{\lpmplmw}\dl}
\newcommand{\glm}{\frac{\lambda}{\lpmplmw}\dm}

\newcommand{\dlpmp}[1]{\frac{#1}{\lpmplmw} }

\newcommand{\ul}{u_{\lambda}}

\newcommand{\um}{u_{\mu}}

\newcommand{\vl}{v_{\la}}
\newcommand{\ue}{u_{\la,\ep}}
\newcommand{\el}[1]{e^{\fk{-\frac{#1}{\lambda}}}}

\newcommand{\modBessel}[2]{I_0\fk{2\sqrt{\gamma\delta{#1}{#2}}}}
\newcommand{\modBesselL}[2]{I_0\fk{2\sqrt{\frac{{#1}{#2}}{\Lambda^2}}}}
\newcommand{\edg}{e^{\delta t+\gamma s}}
\newcommand{\eadbg}[2]{\exp\fk{(\alpha+\delta){#1}+(\beta+\gamma){#2}}}
\newcommand{\eadbgm}[2]{e^{-(\alpha+\delta){#1}-(\beta+\gamma){#2}}}
\newcommand{\dint}{\int_0^t\int_0^s}
\newcommand{\lmw}{\la+\mu-\la\mu\om}

\newcommand{\llmw}{\frac{\la\om-1}{\Lambda}}
\newcommand{\mlmw}{\frac{\mu\om-1}{\Lambda}}

\newcommand{\IO}[1]{I_0\fk{2\sqrt{\frac{#1}{\Lambda^2}}}}

\begin{document}

\title[Asymptotics of Evolution Systems]{Asymptotic Behaviour of Nonlinear Evolution Equations in Banach Spaces}
\author{Josef Kreulich}
\address{Universit\"at Duisburg-Essen}
\date{}
\maketitle
\begin{abstract}
We show how the approach of Yosida approximation of the derivative serves to obtain new results for evolution systems.  Using this method we obtain multivalued time dependent perturbation results.
Additionally, translation invariant subspaces $Y$ of the
bounded and uniformly
continuous functions are considered, to obtain
criteria for the existence of solutions $u\in Y$ to the
equation
$$
u^{\prime}(t)\in 
   A(t)u(t)+ \om u(t) + f(t), t\in \re,
$$ or of solutions $u$ asymptotically close to $Y$ for the
inhomogeneous differential equation
\begin{eqnarray*}
u^{\prime}(t)&\in& A(t)u(t) + \om u(t) + f(t), \ \ t > 0, \\
                  u(0)&=&u_0,
\end{eqnarray*}
 in general Banach spaces,
where $A(t)$ denotes a possibly nonlinear time dependent dissipative operator. Particular examples for the space $Y$ are spaces of
functions with various almost periodicity properties and more
general types of asymptotic behavior. Further, an application to functional differential equations is given.
\end{abstract}

\section{Introduction}
The object of this paper is to study for given $\om \in \re,$ and $A(t)$ dissipative, the abstract evolution equation
\begin{equation} \label{Half-Line-Eq}
u^{\prime}(t)\in A(t)u(t)+\om u(t), t\in \rep,  u(0)=u_o ,
\end{equation} 
the inhomogeneous equation,
\begin{equation} \label{Half-Line-inh-Eq}
u^{\prime}(t)\in A(t)u(t)+\om u(t)+f(t), t\in \rep,  u(0)=u_o,
\end{equation} 
the corresponding equations on $\re,$
\begin{equation} \label{whole-Line-Eq}
u^{\prime}(t)\in A(t)u(t)+\om u(t), t\in \re,
\end{equation} 
and the inhomogeneous evolution equation on $\re,$
\begin{equation} \label{whole-Line-inh-Eq}
u^{\prime}(t)\in A(t)u(t)+\om u(t)+f(t), t\in \re,
\end{equation} 
respectively. Further, an application to functional differential equations is given.

An answer to the question whether the solutions of the approximate equations
\begin{equation} \label{line-Yosida-Approx}
\fk{\frac{d}{dt}}_{\la}u_{\la}(t) \in A(t)u_{\la}(t)+\om \ul(t)+f(t), t \in \re,
\end{equation}
and
\begin{equation}\label{half-line-Yosida-Approx}
\fk{\frac{d}{dt}}_{\la}u_{\la}(t) \in A(t)u_{\la}(t)+\om \ul(t)+f(t), t \in \rep,
\end{equation}
for the whole line problem, or for an initial value problem respectively, converge uniformly
on $\re,$ or for (\ref{Half-Line-inh-Eq}), on $\rep,$ is given.

The main idea using the linear Yosida approximants of the derivative is due to G. Gripenberg, \cite{Gripenberg} and M.G.Crandall and L-C. Evans \cite{CrandEvans}. This method allows to obtain the solution as the uniform limit of fixpoints. Thus, the properties and behaviour of the solution can be obtained from the invariance of the fixpoint mapping. As the fixpoint mapping mainly depends on the resolvent of $A(t),$ the properties of the resolvent carry over to the solution via uniform convergence. \\

Concerning the asymptotic behaviour, the results in \cite{Kreulichbd} are generalized to the nonautonomous case. Moreover, for the asymptotically almost, and the Eberlein weakly almost periodic case we remove the periodicity condition of the evolution operator \cite{KreulichEwap}, and extend it to more general ones. It is shown that the solutions constructed here are integral solutions. Thus, these solutions are in relation with the limit solutions of the equations (\ref{Half-Line-Eq}) and (\ref{Half-Line-inh-Eq}). This comparison leads to results on the asymptotic behaviour of the initial value problems on the half line. Using this comparison we strengthen and generalize results of \cite{AulbachMinh} to uniform convergence on the whole line of the approximants. Consequently our results apply to more general $A(t),$ and to solutions $u$ with not necessarily relatively compact range.

For evolution equations and delay equations, the element $\om\in \re$ plays an essential role. The existence can be interpreted as a type of friction, which implies a loss of energy. In the linear case the boundedness of the solution, or $u\in Y,$ can be derived introducing spectra conditions on operators related to $A(t),$ and $f,$ compare, \cite{RuessPhong}, \cite{BattyHutterRaebiger}, and for a complete discussion \cite{Chicone}.
In the case of linear Integro-Differential-Equations a very general result is given by \cite{Pruess_int}. In the nonlinear case, $\om \in\re$ replaces the absence of the spectra for $A(t)$.

\section{Integral Solutions}

In this section two types of equations are considered, the initial value problem
\begin{equation} \label{int_Half-Line-Eq}
u^{\prime}(t)\in A(t)u(t)+\om u(t), t\in \rep,  u(0)=u_o ,
\end{equation} 
and
the corresponding equation on $\re,$
\begin{equation} \label{int_whole-Line-Eq}
u^{\prime}(t)\in A(t)u(t)+\om u(t), t\in \re.
\end{equation} 
For these equations we find the integral solution. First some prerequisites.

Let $X$ be a general Banach space. As the given $\om \in\re$ plays a crucial role, let $I\in \bk{\rep,\re}$ in the whole paper, $0<\la,\mu <\frac{1}{\btr{\om}}.$ The assumptions for the family $\bk{A(t):t\in I}$ are:

\begin{assu}\label{general-IVA0} 
  The set $\bk{A(t):t \in I}$ is a family of m-dissipative operators.
\end{assu}
\begin{assu}\label{general-IVE1}
  There exist  $h\in BUC(I,X),  $ and $L: \rep \lra \rep$, continuous and monotone non-decreasing,
such that for $\la >0,$ and $t_1,t_2\in I$ we have
\begin{eqnarray*}
\lefteqn{\nrm{x_1-x_2}} \\
&\le& \nrm{x_1-x_2 -\la(y_1-y_2)}+\la\nrm{h(t_1)-h(t_2)}L(\nrm{x_2}),
\end{eqnarray*}
for all $[x_i,y_i]\in A(t_i),$  $i=1,2.$
\end{assu}
\begin{assu}\label{general-IVE2}
There are bounded and Lipschitz continuous functions $g,h:\rep \to X, $ 
and $L: \rep \lra \rep$ continuous, and monotone non-decreasing, 
such that for $\la >0,$ and $t_1,t_2\in I,$ we have
\begin{eqnarray*}
\lefteqn{\nrm{x_1-x_2}} \\
&\le& \nrm{x_1-x_2 -\la(y_1-y_2)}+\la \nrm{h(t_1)-h(t_2)}L(\nrm{x_2}) \\
&&+ \la \nrm{g(t_1)-g(t_2)}\nrm{y_2},
\end{eqnarray*}
for all $[x_i,y_i]\in A(t_i),$  $i=1,2.$
\end{assu}

As the $\om $ plays a crucial role the for evolution equations, $A(t)+\om I$ is discussed as a perturbation of $A(t)$ by $\om I.$
Consequently we split the assumptions on $A(t)$ and $\om $ as well.
To define the integral solution \cite[Definition 6.18 pp. 217-218]{Ito_Kappel}, due to Ph. Benilan \cite{Benilan}, 
we introduce the functions $[\cdot,\cdot]_- $ and $[\cdot,\cdot]_+ $ on $X\times X$ 
with values in 
$\re$ 
similar to \cite[Definition 1.2 p. 4]{Ito_Kappel}
\begin{eqnarray*}
[y,x]_{+} &:=& \lim_{\al \to 0+}\frac{\nrm{x+\al y}-\nrm{x}}{\al} \\
\semim{y,x} &:=& \lim_{\al \to 0+}\frac{\nrm{x}-\nrm{x-\al y}}{\al}.
\end{eqnarray*}
To define the integral solution coming with $A(t)+\om I$ we have to compute the perturbed control functions.
Applying Assumption \ref{general-IVE2} and the triangle inequality, we obtain for 
$(x_i,y_i)\in A(t_i),$ $i=1,2$
\begin{eqnarray*}
\lefteqn{[y_1+\om x_1-y_2-\om x_2,x_1-x_2]_-} \\
&=&[y_1-y_2,x_1-x_2]_-+\om\nrm{x_1-x_2} \\
&\le&\nrm{h(t_1)-h(t_2)}L(\nrm{x_2}) +\btr{\om}\nrm{g(t_1)-g(t_2)}\nrm{x_2}\\
&& +\nrm{g(t_1)-g(t_2)}\nrm{y_2+\om x_2} + \om \nrm{x_1-x_2}.
\end{eqnarray*}
In the case of Assumption \ref{general-IVE1} we assume $g=0.$

The previous observation leads to the perturbed control functions $h^{\om}$ and $L^{\om}.$ 
 
 $$
 \Funk{h^{\om}}{I}{\fk{X\times X,\nrm{\cdot}_1}}{t}{(h(t),\btr{\om}g(t)),}
 $$
and $L^{\om}(t)=L(t)+t.$  In the future, we write $\nrm{\hw(t)}=\nrm{\hw(t)}_1.$ With these functions and definitions we obtain:
\begin{eqnarray*}
\lefteqn{[y_1+\om x_1-y_2-\om x_2,x_1-x_2]_-} \\
&\le& \om\nrm{x_1-x_2}+\nrm{h^{\om}(t_1)-h^{\om}(t_2)}L^{\om}(\nrm{x_2})+\nrm{g(t_1)-g(t_2)}\nrm{y_2+\om x_2},
\end{eqnarray*}
or written in norms, (applying Proposition  \ref{lambda_plus_mu_ineq})
\begin{eqnarray} \label{resolvent_t_stability}
\lefteqn{(1-\la\om)\nrm{x_1-x_2}} \nonumber \\
&\le& \nrm{x_1-x_2-\la(y_1+\om x_1-y_2-\om x_2)} \\
&&+\la \nrm{h^{\om}(t_1)-h^{\om}(t_2)}L^{\om}(\nrm{x_2})+\la\nrm{g(t_1)-g(t_2)}\nrm{y_2+\om x_2}. \nonumber
\end{eqnarray}

Simply, we computed the corresponding stability inequailities of Assumption \ref{general-IVE1} and Assumption \ref{general-IVE2} for $A(t)+\om I,$ and we are ready for the definition of the integral solution.

\begin{defi}
Let $I=\rep,$ and assume that either the Assumption \ref{general-IVE1} or Assumption \ref{general-IVE2} is satisfied for the family $\bk{A(t):t\in I},$ 
$0\le a <b,$
a continuous function 
$u:[a,b]\to X$ is called an integral solution of (\ref{int_Half-Line-Eq}), if $u(a)=u_a,$ and
\begin{eqnarray*} \label{integral-sol-ineq}
\nrm{u(t)-x}-\nrm{u(r)-x}&\le &\int_r^t\fk{[y,u(\nu)-x]_{+}+\om\nrm{u(\nu)-x}}d\nu \\
&&+\Lw(\nrm{x})\int_r^t\nrm{\hw(\nu)-\hw(r)}d\nu +\nrm{y}\int_r^t\nrm{g(\nu)-g(r)}d\nu
\end{eqnarray*}
for all $a\le r\le t\le b,$ and $[x,y]\in A(r)+\om I.$
\end{defi}

In order to solve the initial value problem (\ref{int_Half-Line-Eq}) , we use the Yosida-approximation of the derivative. This leads to the equation

\begin{equation}\label{int_half-line-Yosida-Approx}
\fk{\frac{d}{dt}}_{\la}u_{\la}(t) \in A(t)u_{\la}(t)+\om \ul(t), u(0)=x_0,t \in \rep,
\end{equation}
with the Yosida-approximation,
\begin{equation}\label{defi-Yosida-Approx}
\fk{\frac{d}{dt}}_{\la}u(s):=\frac{1}{\la}\fk{u(s)-u_0-\frac{1}{\la}\int_0^s\nres{\tau}(u(s-\tau)-u_0)d\tau}.
\end{equation}

\begin{theo} \label{IVA0-IVE2-CauchyConvergence}
Let $I=\rep,$ $u_0\in \hat{D},$  $T>0,$ Assumption \ref{general-IVA0} and Assumption \ref{general-IVE2} be fulfilled, then
$\ul $ is locally uniformly Cauchy in $\la \to 0.$ The limit $u(t):=\lim_{\la\to 0}\ul(t),$ is an integral solution to (\ref{int_Half-Line-Eq}) , on $[0,T].$
\end{theo}

\begin{defi}
For $A\subset X\times X,$  dissipative and $\om\in\re$ we define for $\la >0, \btr{\om} < 1/\la$, and $ x\in R(I-\la(A+\om))$:
$$J_{\la}^{\om}x:=(I-\la(A+\om))^{-1}x, \mbox{ and } A_{\la}^{\om}=\frac{1}{\la}(J_{\la}^{\om}-I,)$$
$$ \btr{Ax}:=\lim_{\la\to 0}\nrm{A_{\la}x},$$
and
$$
\hat{D^{\om}}:=\bk{x:\lim_{\la\to 0+}\nrm{A_{\la}^{\om}x}<\infty}.
$$
In the case $\om=0,$ we write $\hat{D^{\om}}=\hat{D}.$
\end{defi}

For operators $A+\om I$ $A$ dissipative we have: 
\begin{remk} \label{perturbed_control}
\begin{enumerate}
\item $\Jlw(t)x=J_{\frac{\la}{1-\la\om}}(t)(\frac{1}{1-\la\om}x).$
\item For all $\om\in\re,$ $\btr{(A+\om I)x}\le\btr{Ax}+\btr{\om}\nrm{x}.$ Consequently , $\hat{D}^{\om}=\hat{D}.$
\item Let Assumption \ref{general-IVA0} and Assumption \ref{general-IVE2} be fulfilled, then using that for  $x\in X $, 
$$
(\Jlw(t)x,A_{\la}^{\om}(t)x)\in A(t)+ \om I
$$ 
the inequality (\ref{resolvent_t_stability}) leads to: 
\begin{eqnarray*}
\lefteqn{\nrm{\Jlw(t)x-\Jlw(s)x}} \\
&\le&\frac{\la}{1-\la\om}\nrm{h^{\om}(t)-h^{\om}(s)}L^{\om}(\nrm{\Jlw(s)x}) \\
&&+\frac{\la}{1-\la\om}\nrm{g(t)-g(s)}\nrm{A_{\la}^{\om}(s)x}.
\end{eqnarray*}
\label{perturbed_control_Jlw}
\item In case of Assumption \ref{general-IVE1}, we have the same inequality, with $g=0,$ and $L^{\om}=L.$
\end{enumerate}
\end{remk}

\begin{remk}\label{hat_D_constant}
The boundedness of the control functions, and  $x\in \hat{D}(A(t_0),$ implies $x\in \hat{D}(A(s),$ for all $s\in I.$ Moreover, due to Assumption \ref{general-IVE1}, and Assumption \ref{general-IVE2}, we have,

$$
\nrm{A_{\la}(t_0)x-A_{\la}(s)x} \le 2\nri{h} L(\nrm{x}),
$$
or
$$
\nrm{A_{\la}(t_0)x-A_{\la}(s)x} \le 2\nri{h} L(\nrm{x})+2\nri{g}\nrm{A_{\la}(t_0)x}),
$$
for all $s \in I, \la >0.$ This leads to $\hat{D}(A(t))=\hat{D},$  and the integral solutions found in Theorems \ref{IVA0-IVE2-CauchyConvergence}, \ref{IVA0-IVE1-CauchyConvergence},\ref{Cauchy-solutions}, will satisfy $\bk{u(t):t\in I}\subset \overline{D}, $ applying that m-dissipativeness implies $D(A(t))\subset \hat{D}(A(t)) \subset \overline{D(A(t))}.$ Use \cite[Lemma 2.4]{Evans}.
\end{remk}

The same we obtain for the Assumption \ref{general-IVE1}:
\begin{theo} \label{IVA0-IVE1-CauchyConvergence}
Let $I=\rep,$ $u_0\in \hat{D},$ $T>0,$ Assumption \ref{general-IVA0} and Assumption \ref{general-IVE1} be fulfilled, then
$\ul$ is locally uniformly Cauchy in $\la \to 0+.$ The limit $u(t):=\lim_{\la\to 0}\ul(t),$ is an integral solution to (\ref{int_Half-Line-Eq}) on $[0,T].$
\end{theo}

In order to solve the whole line problem (\ref{int_whole-Line-Eq}), we have $I=\re,$ and the following definition for the integral solution.
\begin{defi}
Let $I=\re,$ and assume that either the Assumption \ref{general-IVE1} or Assumption \ref{general-IVE2} is satisfied for the family $\bk{A(t):t\in I}.$ 
A continuous function 
$u:\re \to X$ is called an integral solution on $\re$ to (\ref{int_whole-Line-Eq}), if and only if 
\begin{eqnarray*} \label{line-integral-sol-ineq}
\nrm{u(t)-x}-\nrm{u(r)-x}&\le &\int_r^t\fk{[y,u(\nu)-x]_{+}+\om\nrm{u(\nu)-x}}d\nu \\
&&+\Lw(\nrm{x})\int_r^t\nrm{\hw(\nu)-\hw(r)}d\nu +\nrm{y}\int_r^t\nrm{g(\nu)-g(r)}d\nu
\end{eqnarray*}
for all $-\infty < r\le t <\infty ,$ and $[x,y]\in A(r)+\om I.$
\end{defi}

Similar to the initial value case we consider the Yosida-Approximation of the derivative. This leads to
\begin{equation} \label{int_line-Yosida-Approx}
\fk{\frac{d}{dt}}_{\la}u_{\la}(t) \in A(t)u_{\la}(t)+\om \ul(t), t \in \re,
\end{equation}
with
\begin{equation} \label{defii-whole-line-Yosida-Approx}
\fk{\frac{d}{dt}}_{\la}u(t):=\dl\fk{ u(t)-\dl\intO \el{s}u(t-s)ds}.
\end{equation}

For these approximants we have,

\begin{theo} \label{Cauchy-solutions}

Let $I=\re,$ then we have:
\begin{enumerate}
\item If Assumption \ref{general-IVA0}, Assumption \ref{general-IVE1}  and $\om < 0,$ 
are fulfilled, then the Yosida approximants $\fk{u_{\la}:\la >0}$ 
 are Cauchy in \bucrx when $\la \to 0+.$ The limit 
$$u(t):=\lim_{\la\to 0+}\ul(t)$$ 
is an integral solution to (\ref{int_whole-Line-Eq}) on $\re.$
\item If Assumption \ref{general-IVA0} and Assumption \ref{general-IVE2} is fullfilled,
further assume that the Lipschitz constant $L_g$ of $g$ in Assumption \ref{general-IVE2}
is less than $-\om,$ then the Yosida approximants $\fk{u_{\la}:\la >0}$  
are Cauchy in \bucrx when $\la \to 0+ $. The limit 
$$u(t):=\lim_{\la\to 0+}\ul(t)$$ 
is an integral solution to (\ref{int_whole-Line-Eq}) on $\re.$
\end{enumerate}
\end{theo}

In this section we consider evolution equations of the form,
\begin{equation} \label{whole-line-inhomogeneous-equation}
u^{\prime}(t)\in A(t)u(t)+\om u(t)+f(t), t\in \re,
\end{equation}
and 
\begin{equation} \label{half-line-inhomogeneous-equation}
u^{\prime}(t)\in A(t)u(t)+\om u(t)+f(t), u(0)=u_0, t\in \rep.
\end{equation}
Further we will show using the Yosida approximation, that the obtained solutions are integral solutions as introduced by Ph. Benilan \cite{Benilan}, \cite[Definition 6.29, p.232]{Ito_Kappel}. From the mathematical viewpoint these results came out when bringing similar methods as used in the homogeneous case to bear. Consider the time dependent operator
$$
B(t)=A(t)+f(t).
$$

\begin{defi} \label{def-int-sol-inhomo}
Assume that either the Assumption \ref{general-IVE1} or Assumption \ref{general-IVE2} is satisfied for the family $\bk{A(t):t\in I}.$ In the case of Assumption \ref{general-IVE1} choose $g=0.$
\begin{enumerate}
\item Let $I=\rep.$ A continuous function 
$u:[a,b]\to X$ is called an integral solution of (\ref{half-line-inhomogeneous-equation}) if $u(a)=x_a$ and
\begin{eqnarray*} 
\nrm{u(t)-x}-\nrm{u(r)-x}&\le &\int_r^t\fk{[y+f(\nu),u(\nu)-x]_{+}+\om\nrm{u(\nu)-x}}d\nu \\
&&+\Lw(\nrm{x})\int_r^t\nrm{\hw(\nu)-\hw(r)}d\nu +\nrm{y}\int_r^t\nrm{g(\nu)-g(r)}d\nu
\end{eqnarray*}
for all $a\le r\le t\le b,$ and $s\in [a,b],$ $[x,y]\in A(s)+\om I.$
\item Let $I=\re.$ A continuous function 
$u:\re \to X$ is called an integral solution of (\ref{whole-line-inhomogeneous-equation}) on $\re$ if  
\begin{eqnarray*} 
\nrm{u(t)-x}-\nrm{u(r)-x}&\le &\int_r^t\fk{[y+f(\nu),u(\nu)-x]_{+}+\om\nrm{u(\nu)-x}}d\nu \\
&&+\Lw(\nrm{x})\int_r^t\nrm{\hw(\nu)-\hw(r)}d\nu +\nrm{y}\int_r^t\nrm{g(\nu)-g(r)}d\nu
\end{eqnarray*}
for all $-\infty < r\le t <\infty ,$ and $s \in \re,$ $[x,y]\in A(s)+\om I.$
\end{enumerate}
\end{defi}

Similar to the homogeneous case we obtain the following result.

\begin{theo} \label{buc_int_sol}
\begin{enumerate}
\item Let $I=\re,$ $f\in BUC(\re,X),$  and $A(t)$ fulfill Assumption \ref{general-IVA0} and either Assumption \ref{general-IVE1} with $\om<0,$  or  Assumption \ref{general-IVE2} , with $L_g$ the Lipschitz constant of $g,$ $0\le L_g<-\om,$ then equation (\ref{whole-line-inhomogeneous-equation}) has an integral solution on $\re.$  \label{buc_int_sol_re}
\item Let $I=\rep,$ $u_0\in \overline{D},$ $f\in BUC(\re,X),$ and $A(t)$ fulfill Assumption \ref{general-IVA0} and either Assumption \ref{general-IVE1} with $\om <0$, or  Assumption \ref{general-IVE2} , with $L_g$ the Lipschitz constant of $g,$ $0\le L_g<-\om,$ then equation (\ref{half-line-inhomogeneous-equation}) has an integral solution on $\rep.$ \label{buc_int_sol_rep}
\end{enumerate}
\end{theo}

Before we can pass to more general right hand sides we have to prove a few stability estimates.

We consider the stability inequality for solutions on $\re.$

\begin{pro} \label{inhomgeneous-stability}
Let $I=\re,$ and  $A(t)$ fulfill Assumption \ref{general-IVA0}
\begin{enumerate}
\item  Let $\om <0,$ and Assumption \ref{general-IVE1} be fulfilled. For given right hand sides $f_1, f_2 \in BUC(\re,X),$ $u_1,u_2$ the corresponding integral solutions on $\re,$
we have
\begin{eqnarray*}
\nrm{u_1(t)-u_2(t)} &\le& \int_0^{\infty}\exp(\om r)\nrm{f_1(t-r)-f_2(t-r)}dr.
\end{eqnarray*}
\item Let $0\le L_g<-\om,$ $L_g$ the Lipschitz constant of $g,$ and  Assumption \ref{general-IVE2} be fulfilled. For given right hand sides  $f_1, f_2$ bounded and Lipschitz, and $u_1,u_2$ the corresponding integral solutions on $\re,$
then
\begin{eqnarray*}
\nrm{u_1(t)-u_2(t)} &\le& \int_0^{\infty}\exp(\om r)\nrm{f_1(t-r)-f_2(t-r)}dr.
\end{eqnarray*}
\end{enumerate}
\end{pro}

\begin{pro} \label{half-line-inhomgeneous-stability}
Let $I=\rep,$ and $A(t)$ fulfill Assumption \ref{general-IVA0}.
\begin{enumerate}
\item If Assumption \ref{general-IVE1} is fulfilled, for given right hand sides, $f_1, f_2 \in BUC(\re,X),$ and $u_1,u_2$ the corresponding solutions on $\rep$ with the initial values $x_1,x_2,$
\begin{eqnarray*}
\nrm{u_1(t)-u_2(t)} &\le&\exp(\om t)\nrm{x_1-x_2} +\int_0^t\exp(\om (t-\tau))\nrm{f_1(\tau)-f_2(\tau)}d\tau
\end{eqnarray*}
for all $t>0.$
\item If Assumption \ref{general-IVE2} is fulfilled, for given right hand sides, $f_1, f_2$ bounded and Lipschitz, and $u_1,u_2$ the corresponding solutions on $\rep,$ with the initial values $x_1,x_2,$
then
\begin{eqnarray*}
\nrm{u_1(t)-u_2(t)} &\le&\exp(\om t)\nrm{x_1-x_2} +\int_0^t\exp(\om (t-\tau))\nrm{f_1(\tau)-f_2(\tau)}d\tau .
\end{eqnarray*}
\end{enumerate}
\end{pro}

For the comparison between the whole line and the half line problem we have:

\begin{cor} \label{compare-half-whole-line}
Let  $A(t)$ fulfill Assumption \ref{general-IVA0} and either Assumption \ref{general-IVE1}, with $f\in BUC(\re,X)$, or \ref{general-IVE2}, with $L_g$ the Lipschitz constant of $g,$ $0\le L_g<-\om,$ and $f\in \bucrx.$ Then the solution $u$ of (\ref{whole-line-inhomogeneous-equation}) and the solution $v$ of (\ref{half-line-inhomogeneous-equation}) satisfy
\begin{equation*}
\nrm{u(t)-v(t)}\le \exp(\om t)\nrm{u(0)-x_0}
\end{equation*}
for all $ 0\le s\le t.$ 
\end{cor}

The comparsion between the whole and the half line leads to:

\begin{theo}\label{uniform_convergence_half_line}
Let $I=\rep,$ and $A(t)$ fulfill Assumption \ref{general-IVA0} and either Assumption \ref{general-IVE1}, with $f\in BUC(\rep,X),$ or Assumption \ref{general-IVE2} with $f\in\bucrpx,$ and $L_g$ the Lipschitz constant of $g,$ $0\le L_g<-\om.$
Further let $u_0\in \overline{D},$ and  $\ul$ the corresponding Yosida-approximations to equation (\ref{int_half-line-Yosida-Approx}). Then $\ul$ converge unifomly on $\rep$ to an integral solution of (\ref{int_Half-Line-Eq}).
\end{theo}

\begin{theo}\label{uniform-whole line-convergence}
Let $I=\re,$ and $A(t)$ fulfill Assumption \ref{general-IVA0} and either Assumption \ref{general-IVE1}, with $f\in \bucrx,$ or Assumption \ref{general-IVE2} with $f\in\bucrx,$ and $L_g$ the Lipschitz constant of $g,$ $0\le L_g<-\om.$
Further let $\ul$ the corresponding Yosida-approximations to equation (\ref{int_line-Yosida-Approx}). Then $\ul$ converge unifomly on $\re$ to an integral solution of (\ref{int_whole-Line-Eq}).
\end{theo}

In the next theorems we apply the method of construction of the approximants for equations (\ref{int_half-line-Yosida-Approx}) and (\ref{int_whole-Line-Eq}).

\begin{theo} \label{solution-in-Y}
Let $I=\re.$
\begin{enumerate}
\item Let $Y\subset BUC(\re,X)$ be a closed and translation invariant subspace, such that 
$$\bk{t\mapsto \Jl(t)f(t)} \in Y, \mbox{ for all }f\in Y.$$
Further, let  $A(t)$ fulfill Assumption \ref{general-IVA0}, and  Assumption \ref{general-IVE1} with $\om <0.$ Then equation (\ref{int_whole-Line-Eq}) has an integral solution $u\in Y$ on $\re.$
\item Let $Y\subset \bucrx$ be a closed and translation invariant subspace,such that 
$$\bk{t\mapsto \Jl(t)f(t)} \in Y,\mbox{ for all }f\in Y.$$ 
Further, let  Assumption \ref{general-IVE2} with  $L_g$ the Lipschitz constant of $g,$ $0\le L_g<-\om.$ Then equation (\ref{int_whole-Line-Eq}) has an integral solution $u\in Y$ on $\re.$
\end{enumerate}
\end{theo}

As in \cite{AulbachMinh} only the case of Assumption \ref{general-IVE1} is considered; we can extend the result to the case of Assumption \ref{general-IVE2}. Using $Y=\apx,$ we get:

\begin{theo}\label{solution-in-APRX}
Let $\bk{t\mapsto \Jl(t)x}\in \apx$ for all $x\in X, 0<\la\le 1/\btr{\om}.$  
Further, let  $A(t)$ fulfill Assumption \ref{general-IVA0} and either Assumption \ref{general-IVE1} with $\om <0,$ or Assumption \ref{general-IVE2} with $L_g$ the Lipschitz constant of $g,$ $0\le L_g<-\om.$ Then equation (\ref{whole-Line-inh-Eq}) has an integral solution $u\in \apx$ on $\re.$
\end{theo}

\begin{remk} Similar results are abtained for periodic, antiperiodic, asymptotically almost periodic, weakly almost periodic(with or without compact range), and other types of almost periodicity, continuous almost automorphy, and ergodicity, which lead to a closed translation invariant subspace, or closed translation invariant cone of $\bucrx$, which is invariant with respect to $\Jl(\cdot), \forall 0<\la<1/\btr{\om}.$
\end{remk}

In the previous result the uniform continuity is necessary, and as proved by \cite{Veech} continuous almost automorphic functions are. S. Bochner introduced the notion of almost automorphic functions which are not necessarily uniformly continuous, but to such functions the method applies as well.
\begin{defi}
A function $f\in C(\re,X)$ is said to be almost automorphic if for any $\bk{s_n}_{n\in\za},$ there is a subsequence $\bk{s_{n_k}}_{k\in\za}$ such that
$$
\lim_{k\to\infty}f(t+s_{n_k})=g(t), \forall t\in \re,
$$
and
$$
\lim_{k\to\infty}g(t-s_{n_k})=f(t), \forall t\in \re.
$$
We define 
$$AA(\re,X)=\bk{f\in C(\re,X): f \mbox{ almost automorphic}}.$$
\end{defi}

We will show that the property to be almost automorphic carries over from the control function.
\begin{pro}\label{solution-in-AARX}
Let  $A(t)$ fulfill Assumption \ref{general-IVA0} and Assumption \ref{general-IVE1} with $\om <0,$ and a control function $h\in AA(\re,X).$
Then equation (\ref{int_whole-Line-Eq}) has an integral solution $u\in AA(\re,X).$
\end{pro}

\begin{remk}[Used Method and Comparison to existing results] 
The method applied is mainly used to obtain general asymptotics of the time dependent initial value problem.
As far as existence is concerned, there are partly more general results by different approaches, see below. The main point of the approach via the Yosida-approximants of the derivative chosen here is to deduce both existence and at the same time general results on the asymptotic behaviour of solutions, such as Thm. \ref{solution-in-Y}, Thm. \ref{solution-in-APRX} and Prop. \ref{solution-in-AARX}.
 \begin{enumerate}
\item In the view of existence the Assumption \ref{general-IVE1} is quite similar to 
the one in the Crandall Pazy paper, \cite[(C.1), p. 62]{CrandallPazy}, but Assumption \ref{general-IVE1} needs all pairs $(x,y)\in A(t)$ not only $(J_{\la}(t)x, A_{\la}(t)x)$. 
\item Assumption \ref{general-IVE2} is stronger than the one in 
Crandall's and Pazy's paper \cite[(C.2), p. 62]{CrandallPazy} or in the textbook of Ito and Kappel \cite{Ito_Kappel}, but nearly identical to the one of Kato \cite[(3.1), p. 513]{TKato}, who proved existence in the case $X$ uniformly convex. Comparing with actual assumptions where $g,h$ is  only continuous and of bounded variation \cite[(E2), p. 187]{Ito_Kappel} its Lipschitz continuity strengthens the assumption. This stronger assumption allows to obtain the
 uniform convergence of the half-line and whole-line problem, even if a dependency on 
$\nrm{y_2}$ is given. 
\item In the case of Assumption \ref{general-IVE2} the problem of  boundedness of the 
$\fk{\frac{d}{dt}}_{\la}u_{\la}$ occurs. This lead to the Lipschitz continuity of the control functions $h,g.$ For this 
reason the assumption is strengthened by linearizing the dependency on $\nrm{y_2}. $ Remember, Crandall/Pazy, or Ito/Kappel use $\nrm{f(t_1)-f(t_2)}L(\nrm{x_2})(1+\nrm{y_2}).$ 
 In Assumption \ref{general-IVE2}, the dependency on $ \nrm{x_2}$ and $\nrm{y_2}$ is splitted, this shows that in the  prerequisites containing assumption 2.2 the 
Lipschitz constant of $g$ comes into play. 
\item When looking for general existence we refer to \cite{CrandallPazy}, \cite{Evans}, or to the textbook \cite{Ito_Kappel}, but in this paper we give an approach 
to obtain asymptotics which is coming up with the approximants. 
\item In case of a single
valued operator  boundedness is proven in \cite{bd-Kartsatos}, periodicity in  \cite{litcanu}.
\end{enumerate}
\end{remk}

\section{Proofs of Technical Preliminaries in the Half Line Case}

\begin{proof}[Proof of \ref{perturbed_control}:]  The first identity is a simple computation. For the second inequality note that,
$$
\btr{(A+\om)x}=\lim_{\la\to 0}\nrm{A^{\om}_{\la}x}=\lim_{\la\to 0}\frac{1}{\la}\nrm{\Jlw x-x}, $$
and
\begin{eqnarray*}
\frac{1}{\la}\nrm{(\Jlw-I)x}
&\le& \frac{1}{\la}\nrm{J_{\frac{\la}{1-\la\om}}x-x}+\frac{1}{\la}\nrm{J_{\frac{\la}{1-\la\om}}x-J_{\frac{\la}{1-\la\om}}(\frac{1}{1-\la\om}x)}\\
&\le&\nrm{A_{\frac{\la}{1-\la\om}}x}+\frac{\btr{\om}}{1-\la\om}\nrm{x}.
\end{eqnarray*}
\end{proof}

Starting with the initial value problem, we have to find the $\ul$ solving the approximate equation. Rearranging of (\ref{int_half-line-Yosida-Approx}) leads to a fixpoint equation. Thus, for small $\la>0,$ we have to solve the fixpoint equation for
$$
\Funk{F_{\la}}{C([0,T],X)}{C([0,T],X)}{u}{\bk{t\mapsto \Jlw(t)\fk{ \nres{t} x_0+\frac{1}{\la} \int_0^t\nres{\tau}\ul(t-\tau)d\tau}}}.
$$

\begin{lem}
Let $I=\rep,$ and $A(t)$ be $\om$-m-dissipative for some $\om \in \re.$ Then $F_{\la}$ has a fixpoint for all $\la < \frac{1}{\btr{\om}}.$
\end{lem}

\begin{proof} Defining
$$
\Funk{S}{C[0,T]}{C[0,T]}{u}{\bk{t\mapsto \frac{1}{\la(1-\la\om)}\int_0^t\nres{t-\tau}u(\tau)d\tau},}
$$
for $F_{\la},$ we have,
\begin{eqnarray*}
\lefteqn{\nrm{F_{\la}u(t)-F_{\la}v(t)} } \\
&=&\nrm{\Jlw(t)\fk{ \nres{t} x_0+\frac{1}{\la} \int_0^t\nres{\tau}u(t-\tau)d\tau} 
- \Jlw(t)\fk{ \nres{t} x_0+\frac{1}{\la} \int_0^t\nres{\tau}v(t-\tau)d\tau}}, \\
&\le&\frac{1}{\la(1-\la\om)}\int_0^t\nres{t-\tau}\nrm{u(\tau)-v(\tau}d{\tau}, \\
&\le& \nri{S\bk{t\mapsto\nrm{u(t)-v(t)}}}.
\end{eqnarray*}
As we may take the sup on the left hand side , the fact that $S$ is quasi-nilpotent serves for a strict contraction $F_{\la}^n$ for some $n,$ which leads to a fixpoint of $F_{\la}.$
\end{proof}

\begin{lem} \label{IVE2-half-line-boundedness-lemma}
Let $I=\rep,$ Assumption \ref{general-IVA0} and \ref{general-IVE2} be fulfiled, $T>t>0,$ and $\ul$ the solution to (\ref{int_half-line-Yosida-Approx}). Further, let $u_0 \in \hat{D},$ for all $t\in[0,T] $, i.e. $\nrm{A_{\la}(t)u_0}\le C_1,$ for all $0< \la < 1/\btr{\om},$ and $0\le t\le T,$ then
\begin{enumerate}
\item $\ul(t)$ is uniformly bounded for $0< \la < 1/\btr{\om},$ and $0\le t\le T.$ Moreover,
$$
\nrm{\ul(t)-u_0}\le K^{\prime}(t+\la),
$$ 
for some constant $K^{\prime} \ge 0.$
\item Let $L_g$ be the Lipschitz-constant for the control function $g$. Then, for $0< \la \le c < \frac{1}{\btr{\om}+L_g}$ the family
$$ D_c:=\bk{\ul:[0,T]\to X:0<\la<c},$$ 
is equi-Lipschitz.Consequently, the Yosida approximation
$$
\frac{1}{\la}\fk{\ul(t)-u_0-\frac{1}{\la}\int_0^t\nres{\tau}(\ul(t-\tau)-u_0)d\tau}
$$
is uniformly bounded for $0< \la \le c,$ and $0\le t\le T.$ 
\end{enumerate}
\end{lem}

\begin{proof}
For the proof of the boundedness of $\ul(t)$ we show the boundedness of $\nrm{\ul(t)-u_0}.$
Let $u_0\in\hat{D}$ and $\nrm{\Jlw(t)u_0-u_0}\le \la C_1.$
\begin{eqnarray*}
\nrm{\ul(t)-u_0}
&=&\nrm{\Jlw(t)\fk{\nres{t}u_0+\frac{1}{\la} \int_0^{t}\nres{\tau}\ul(t-\tau)d\tau}-u_0}, \\
&\le& \nrm{\Jlw(t)\fk{\nres{t}u_0+\frac{1}{\la} \int_0^{t}\nres{\tau}\ul(t-\tau)d\tau}-\Jlw(t)u_0} \\
&&+\nrm{\Jlw(t)u_0-u_0},\\
&\le & \frac{1}{\la(1-\la\om)}\int_0^t\nres{\tau}\nrm{\ul(t-\tau)-u_0}d\tau +\la C_1.
\end{eqnarray*}

Thus, by \ref{s-t-integral-inequality}, we have,
\begin{eqnarray} \label{ul-bd-if-w-less-0}
\nrm{\ul(t)-u_0}
&\le& \la C_1 +\frac{1}{1-\la \om}\int_0^t \exp\fk{\frac{\om}{1-\la\om} \tau}C_1 d\tau \\
&\le & K^{\prime}(\la + t) \nonumber
\end{eqnarray}

The proof that $\ul$ is uniformly Lipschitz splits into four steps.
The first step is the derivation of the preliminary inequality for $\nrm{\ul(t-s)-\ul(t)}.$
We start with the definition of $\ul$ as a solution to (\ref{half-line-Yosida-Approx}). Thus, by the inequality (\ref{resolvent_t_stability}) we find with the half line Yosida approximation
\begin{eqnarray*} 
\lefteqn{\nrm{\ul(t-s)-\ul(t)}} \\
&\le& \frac{1}{1-\la\om}\nrm{\ul(t-s)-\ul(t)-\la\bk{\fk{\frac{d}{dt}}_{\la}\ul(t-s)-\fk{\frac{d}{dt}}_{\la}\ul(t)}} \\
&&+\frac{\la}{1-\la\om} \nrm{h^{\om}(t-s)-h^{\om}(t)}L^{\om}(\nrm{\ul(t)}) \\
&&+\frac{\la}{1-\la\om}\nrm{g(t-s)-g(t)} \nrm{\frac{1}{\la}\fk{\ul(t)-u_0-\int_0^t\nresl{\tau}(\ul(t-\tau)-u_0)d\tau}},\\
&\le& \frac{1}{1-\la\om} 
\nrm{\nres{t-s}u_0+\int_0^{t-s}\nresl{\tau}u(t-s-\tau)d\tau
   - \nres{t}u_0-  \int_0^{t}  \nresl{\tau}u(t-\tau)d\tau}  \\
&&+\frac{\la}{1-\la\om} \nrm{h^{\om}(t-s)-h^{\om}(t)}L^{\om}(\nrm{\ul(t)}) \\
&&+\frac{\la}{1-\la\om}\nrm{g(t-s)-g(t)} \nrm{\frac{1}{\la}\fk{\ul(t)-u_0-\int_0^t\nresl{\tau}(\ul(t-\tau)-u_0)d\tau}},\\
&\le&\frac{1}{\la(1-\la\om)}\int_0^{t-s}\nres{\tau}\nrm{\ul(t-s-\tau)-\ul(t-\tau)}d\tau \\
&&+\frac{1}{\la (1-\la\om)}\int_{t-s}^t\nres{\tau}\nrm{\ul(t-\tau)-u_0}d\tau  \\
&&+\frac{\la}{1-\la\om} \nrm{h^{\om}(t-s)-h^{\om}(t)}L^{\om}(\nrm{\ul(t)})\\
&&+\frac{\la}{1-\la\om}\nrm{g(t-s)-g(t)} \nrm{\frac{1}{\la}\fk{\ul(t)-u_0-\int_0^t\nresl{\tau}(\ul(t-\tau)-u_0)d\tau}}.
\end{eqnarray*}

The second step is to prove the boundedness and integrability of the help functions
\begin{eqnarray*}
K_{\la}(t)&:=& \limsup_{s\to 0+}\frac{1}{s}\nrm{\ul(t-s)-\ul(t)}\\
K^{\sup}_{\la}(t)&:=&\sup_{0<\al<t}K_{\la}(\al).
\end{eqnarray*}
Therefore, we fix $\la>0,$ and note the estimates 
$$\nrm{\frac{1}{\la}\fk{\ul(t)-u_0-\int_0^t\nresl{\tau}(\ul(t-\tau)-u_0)d\tau}}\le \frac{2}{\la}\nri{\ul-u_0},$$
and
\begin{eqnarray} \label{t-s-to-t-integral-ul-lipschitz}
\frac{1}{s\la}\int_{t-s}^t\nres{\tau}\nrm{\ul(t-\tau)-u_0}d\tau
&\le& \frac{K^{\prime}}{\la s} \int_{t-s}^t\nres{\tau}(\la+(t-\tau))d\tau, \nonumber \\
&=&  \frac{K^{\prime}}{\la s} \int_0^s\nres{t-\tau}(\la+\tau)d\tau, \nonumber\\
&=& K^{\prime}\frac{\nres{t}}{s\la}\fk{\la\int_0^se^{\frac{\tau}{\la}}d\tau
+\la se^{\frac{s}{\la}} -\la\int_0^se^{\frac{\tau}{\la}}d\tau}, \nonumber \\
&=&K^{\prime}\nres{t-s}.
\end{eqnarray}
Thus, the first step lead to,
\begin{eqnarray*}
\lefteqn{\frac{\exp(t/\la)}{s}\nrm{\ul(t-s)-\ul(t)}}\\
&\le&\frac{1}{\la(1-\la\om)}\int_s^t\frac{\exp(\tau/\la)}{s}\nrm{\ul(\tau-s)-\ul(\tau)}d\tau+C_1^{\la}.\\
\end{eqnarray*}
The well known Gronwall/Bellmann inequality gives
\begin{eqnarray*}
\frac{\exp(t/\la)}{s}\nrm{\ul(t-s)-\ul(t)}&\le&C_1^{\la}\exp\fk{\frac{1}{\la(1-\la\om)}t},
\end{eqnarray*}
which proves boundedness of the help functions. Next we prove the measurability. 
As 
$$
\limsup_{s\to 0+}\frac{1}{s}\nrm{u_{\la}(t-s)-\ul(t)} =\inf_{s>0}\sup_{0<\sigma<s}\frac{1}{\sigma}\nrm{u_{\la}(t-\sigma)-\ul(t)},
$$
we define,  
$$
\bk{q_m}_{m \in\za}:=(0,T]\cap\rat. 
$$
By the continuity of $\ul$ and the norm, we can define the obviously measurable sequence of functions:
$$
\Funk{f_n}{[0,T]}{\rep}{t}{\sup_{q_m<\frac{1}{n}}\frac{1}{q_m}\nrm{\ul(t-q_m)-\ul(t)}.}
$$
Again, the continuity gives,
$$
\inf_{s>0}\sup_{0<\sigma<s}\frac{1}{\sigma}\nrm{\ul(t-\sigma)-\ul(t)} = \inf_{n\in\za}\sup_{q_m<\frac{1}{n}}\frac{1}{q_m}\nrm{\ul(t-q_m)-\ul(t)}
$$
which proves the measurability of $K_\la(\cdot).$ Due to the monotonicity of $K_{\la}^{\sup}$ its measureablility is straightforward. The boundedness of the help functions serves for the local integrability, which closes step 2.

In step 3, we will refine the estimates for $K_{\la}(t),$  and $K_{\la}^{\sup}(t).$
To estimate 
$$
\frac{1}{\la}\fk{\ul(t)-u_0-\int_0^t\nresl{\tau}(\ul(t-\tau)-u_0)d\tau},
$$
we start with $t=0.$ By the first item of the Lemma, we have $\nrm{\frac{1}{\la}(\ul(0)-u_0)} \le \frac{1}{\la}K^{\prime}(\la+0) \le K^{\prime}.$  Due to step 2 we have for $0<s<t,$
$\frac{1}{s}\nrm{\ul(t-s)-\ul(t)}$ is bounded for fixed $\la>0.$
We have,
\begin{eqnarray*}
\frac{1}{\la}\nrm{\ul(t)-u_0-\frac{1}{\la}\int_0^t\nres{\tau}(\ul(t-\tau)-u_0)d\tau}
&\le& \frac{1}{\la^2}\int_0^t\nres{\tau}\nrm{\ul(t-\tau)-\ul(t)}d\tau \\
&&+ \frac{1}{\la}\nres{t}\nrm{\ul(t)-u_0}.
\end{eqnarray*}
The second term is quite simple to estimate. Note,

\begin{eqnarray*}
\frac{1}{\la}\nres{t}\nrm{\ul(t)-u_0}
&\le& \frac{1}{\la}\nres{t}K^{\prime}(\la+t ) ,\\
&\le & K^{\prime}\nres{t} +K^{\prime}\frac{t}{\la}\nres{t}, \\
&\le& C^{\prime}.
\end{eqnarray*}
Appling Proposition \ref{bd-limsup-Lipschitz} for given $t>0,$ we have
$$
\sup_{0<s<t}\frac{1}{s}\nrm{\ul(t-s)-\ul(t)} \le K^{\sup}_{\la}(t).
$$

Hence, for the integral we obtain
\begin{eqnarray*}
\lefteqn{\frac{1}{\la^2}\int_0^t\nres{\tau}\nrm{\ul(t-\tau)-\ul(t)}d\tau} \\
&\le& \frac{1}{\la^2}\int_0^{t}\nres{\tau}\tau d\tau \sup_{0<s<t}\frac{1}{s}\nrm{\ul(t-s)-\ul(t)}d\tau, \\ 
&\le& K^{\sup}_{\la}(t). 
\end{eqnarray*}

We conclude
\begin{eqnarray*}
\frac{1}{\la}\nrm{\ul(t)-u_0-\frac{1}{\la}\int_0^t\nres{\tau}(\ul(t-\tau)-u_0)d\tau}
&\le& K^{\sup}_{\la}(t)+C^{\prime}.
\end{eqnarray*}

Let $L_g,L_h$ be the corresponding Lipschitz constants for $g,h^{\om},$ $L(\nrm{\ul(t)})\le C_u,$ and $K$ from the previous inequality. The inequality derived in step 1 for small $\la>0,$ and $s\le t\le T$  reduces to

\begin{eqnarray} \label{K_lambda_ineq_1}
\lefteqn{(1-\la\om)\frac{1}{s}\nrm{\ul(t-s)-\ul(t)}} \nonumber \\
&\le& \frac{1}{\la}\int_0^{t-s}\nres{\tau}\frac{1}{s}\nrm{\ul(t-s-\tau)-\ul(t-\tau)}d\tau 
+K^{\prime}\exp\fk{\frac{s-t}{\la}} \nonumber \\
&&+\la (L_hC_u+L_gK+L_gC^{\prime}) +\la L_g K^{\sup}_{\la}(t). 
\end{eqnarray}

Applying $\limsup_{s\to 0+}$ on both sides of (\ref{K_lambda_ineq_1}), we can apply Fatou's Lemma to the integral, and obtain for $K_{\la}(t)$ the inequality

\begin{eqnarray*}
(1-\la\om)K_{\la}(t) &\le& \frac{1}{\la}\int_0^t\nres{\tau}K_{\la}(t-\tau)d\tau \\
&&+K^{\prime}\nres{t} +\la C_2+\la L_gK_{\la}^{\sup}(t).
\end{eqnarray*}

Therefore, we are in the situation to apply Lemma \ref{s-t-integral-inequality}. This leads to,
\begin{eqnarray*} 
\lefteqn{K_{\la}(t)}  \\
&\le & \frac{1}{1-\la\om} K^{\prime}\exp(-\frac{t}{\la}) +\la C_2 +\la L_gK_{\la}^{\sup}(t)  \\
&&+\frac{1}{\la(1-\la\om)}\int_0^t\exp\fk{\frac{\om}{1-\la\om}(t-\tau)}\fk{K^{\prime}\exp(-\frac{\tau}{\la}) +\la(C_2+L_g K_{\la}^{\sup}(\tau))}d\tau.
\end{eqnarray*}

Note,
\begin{eqnarray*}
\int_0^t\exp\fk{\frac{\om}{1-\la\om}(t-\tau)}K^{\prime}\exp(-\frac{\tau}{\la})d\tau
&\le & C_3(1-\la\om)\la.
\end{eqnarray*}

Thus, we have for a constant $C_3,$

\begin{eqnarray*}
\lefteqn{K_{\la}(t)}  \\
& \le & \frac{1}{1-\la\om} \fk{C_3 +\la L_g K_{\la}^{\sup}(t)} \\
&& +\frac{L_g}{1-\la\om}\int_0^t \exp \fk{ \frac{\om}{1-\la\om} (\tau) } K_{\la}^{\sup}(t-\tau)d\tau.%
\end{eqnarray*}

Due to the positivity and monotonicity of $K_{\la}^{\sup}$ the right hand side is monotone increasing. We obtain for 
$\la <\frac{1}{\btr{\om}+L_g},$
\begin{eqnarray*}
\lefteqn{K_{\la}^{\sup}(t)} \\
&\le & \frac{1}{1-\la(\om+L_g)}C_3  \\
&&+\frac{L_g}{1-\la(\om+L_g)}\int_0^t\exp\fk{\frac{\om}{1-\la\om}(t-\tau)}K_{\la}^{\sup}(\tau)d\tau.
\end{eqnarray*}
Again an application of Lemma \ref{s-t-integral-inequality} gives
for 
\begin{eqnarray*}
\gamma_{\la}&:=&\frac{\om+L_g-\la\om(\om+2L_g)}{(1-\la\om)(1-\la(\om+L_g)}\\
&\lra&L_g+\omega,
\end{eqnarray*}
the inequality
\begin{eqnarray} \label{w-less-L_g-equiLipschitz}
\lefteqn{K_{\la}^{\sup}(t)} \nonumber \\
&\le&\frac{1}{1-\la(\om+L_g)}C_3 +\frac{L_g}{1-\la(\om+L_g)}\int_0^t\exp(\gamma_{\la}(t-\tau))C_3d\tau.
\end{eqnarray}
This implies that $K_{\la}^{\sup}(t)$ is uniformly bounded for small  $\la >0,$ more precisly for
$$
0<\la\le c < \frac{1}{\btr{\om}+L_g}, 
$$
$$
D_c:=\bk{\fk{\frac{d}{dt}}_{\la}\ul(t):t\in[0,T],0<\la\le c }
$$ 
is bounded.
\end{proof}

\begin{cor} \label{corollary-ul-bd-if-w-less-0}
Let $I=\rep,$ Assumption \ref{general-IVA0} and Assumption \ref{general-IVE1} or \ref{general-IVE2} be fulfilled,  $u_0\in \hat{D}$ and $\ul$ the solution to (\ref{int_half-line-Yosida-Approx}). If $\om <0,$ then $\ul $ is uniformly bounded for $0<\la<\ 1/\btr{\om}.$
\end{cor}
\begin{proof} Apply the inequality (\ref{ul-bd-if-w-less-0})\end{proof}

\begin{cor}
Let $I=\rep,$ Assumption \ref{general-IVA0} and Assumption \ref{general-IVE2} with $0\le L_g<-\om$ be fulfilled, and $\ul$ a solution to (\ref{int_half-line-Yosida-Approx}) with $u_0\in\hat{D}$ on $I$, then $\ul$ is uniformly equi-Lipschitz on $I$.
\end{cor}

\begin{proof}
 From Corollary \ref{corollary-ul-bd-if-w-less-0} we obtain the boundedness of $\ul,$ which implies $L(\nrm{\ul(t)})\le C_u$, consequently we can apply the inequality  (\ref{w-less-L_g-equiLipschitz}).
\end{proof}

\begin{cor} \label{IVE1-half-line-boundeness-cor}
Let $I=[0,T],$ Assumption \ref{general-IVA0} and Assumption \ref{general-IVE1} be fulfilled, $T>t>0,$ and $\ul$ the solution to (\ref{int_half-line-Yosida-Approx}) , then we have:
\begin{enumerate}
\item Let $u_0 \in \hat{D}$  , i.e. $\nrm{A_{\la}(t)u_0}\le C_1,$ for all $0< \la < 1/\btr{\om}$ and $0\le t\le T,$  then $\ul(t)$ is uniformly bounded in $0< \la < 1/\btr{\om},$ and $0\le t\le T.$ Moreover,
$$\nrm{\ul(t)-u_0}\le K^{\prime}(t+\la),$$ for some $K^{\prime}.$
\item 
The family $\bk{\ul:0< \la < 1/\btr{\om}} $ is uniformly equicontinuous on $[0,T].$
\end{enumerate}
\end{cor}
\begin{proof}
 As $g=0,$ we only have to prove the local uniform equicontinuity of $\ul.$
We have:
\begin{eqnarray} \label{A22-equicontinuity_ineq}
(1-\la\om)\nrm{\ul(t-s)-\ul(t)} 
&\le& \frac{1}{\la}\int_0^{t-s}\nres{\tau}\nrm{\ul(t-s-\tau)-\ul(t-\tau)}d\tau \nonumber\\
&&+\frac{1}{\la}\int_{t-s}^t\nres{\tau}\nrm{\ul(t-\tau)-u_0}d\tau \nonumber \\
&&+\la C_2\nrm{h(t-s)-h(t)} .
\end{eqnarray}
From (\ref{t-s-to-t-integral-ul-lipschitz}) we learn,
$$
\frac{1}{\la}\int_{t-s}^t\nres{\tau}\nrm{\ul(t-\tau)-u_0}d\tau \le s K^{\prime}e^{-\frac{t-s}{\la}}
$$
To apply Lemma \ref{s-t-integral-inequality}, we have to subsitute $\nu:=t-s$ in the integral inequality (\ref{A22-equicontinuity_ineq}), i.e.
\begin{eqnarray*}
(1-\la\om)\nrm{\ul(\nu)-\ul(\nu+s)} 
&\le& \frac{1}{\la}\int_0^{\nu}\nres{\tau}\nrm{\ul(\nu-\tau)-\ul(\nu+s-\tau)}d\tau \\
&&+K^{\prime}se^{-\frac{\nu}{\la}}  \\
&&+\la C_2\nrm{h(\nu)-h(\nu+s)} .
\end{eqnarray*}
An application of Lemma \ref{s-t-integral-inequality} gives
\begin{eqnarray}
\lefteqn{\nrm{\ul(\nu)-\ul(\nu+s)}}\nonumber\\
&\le& K^{\prime}se^{-\frac{\nu}{\la}} +\la C_u\nrm{h(\nu)-h(\nu+s)}\nonumber \\
&&+\frac{1}{\la(1-\la\om)}\int_0^{\nu}e^{\frac{\om \tau}{1-\la\om}}
\fk{K^{\prime}se^{-\frac{\nu-\tau}{\la}} +\la C_u\nrm{h(\nu-\tau)-h(\nu-\tau+s)}}d\tau.
\end{eqnarray}
As 
\begin{eqnarray*}
\lefteqn{\frac{1}{\la(1-\la\om)}\int_0^{\nu}e^{\frac{\om \tau}{1-\la\om}}e^{-\frac{\nu-\tau}{\la}}d\tau} \\
&=&\exp\fk{\frac{\om\nu}{1-\la\om}}=\exp\fk{\frac{\om(t-s)}{1-\la\om}},
\end{eqnarray*}
the claim is proved. 
\end{proof}

\begin{cor}
Let $I=\rep,$ Assumption \ref{general-IVA0} and Assumption \ref{general-IVE1} with $\om <0$ be fulfilled, and $\ul$ a solution to (\ref{int_half-line-Yosida-Approx}) on $I$, then $\ul$ is uniformly equicontiuous on $I$.
\end{cor}

\begin{proof}
 From Corollary \ref{corollary-ul-bd-if-w-less-0} we obtain $\ul$ uniformly bounded, consequently we can apply (\ref{A22-equicontinuity_ineq}).
\end{proof}

To obtain boundedness for the $\ul$ uniformly in $\la>0,t\ge 0$ in the case of $\om=0,$ we have:

\begin{pro} \label{omega-zero-bounded-Yosida-approximation}
Let $I=\rep,$ $\om=0,$ and Assumption \ref{general-IVA0} hold. 
\begin{enumerate}
\item For $N\subset \rep$ a set of zero measure let 
$ x_0\in \bigcap_{t\in\rep, t\not\in N}D(A(t)), $ 
$f\in L^1(\rep,X),$ s.t. $f(t)\in A(t)x_0, t\not\in N.$ 
then if either $f$ is bounded, or $\bk{\ul}_{\la>0}$ is equicontinous, $\ul(t)$ is uniformly bounded in $\la>0,t\ge 0.$ 
\label{om_gleich_0_bd}
\item If for an $x_0\in\hat{D},$ $f\in L^1(\rep),$ and $\btr{A(t)x_0}\le f(t)$ a.e.,
then if either $f$ is bounded, or $\bk{\ul}_{\la>0}$ is equicontinous, then $\ul(t)$ is uniformly bounded in $\la>0,t\ge 0.$
\label{om_gleich_0_bd_hat_D}
\end{enumerate}
\end{pro}

\begin{proof}
We start with (\ref{om_gleich_0_bd}):
 Note that, a.e.,
\begin{eqnarray*}
0&\ge&[\dl\fk{\ul(t)-u_0-\dl\int_0^t\el{\tau}(\ul(t-\tau)-u_0)d\tau}-f(t),\ul(t)-x_0]_- \\
&\ge&[\dl\fk{\ul(t)-x_0-\dl\int_0^t\el{\tau}(\ul(t-\tau)-x_0)d\tau}+\dl\el{t}(x_0-u_0)-f(t),\ul(t)-x_0]_-.
\end{eqnarray*}
Rearranging the inequality, we derive
\begin{eqnarray*}
\nrm{\ul(t)-x_0}&\le& \la \nrm{f(t)}+\el{t}\nrm{x_0-u_0} +\dl\int_0^t\el{s}\nrm{\ul(t-s)-x_0}ds.
\end{eqnarray*}
a.e. in $t.$
An application of the continuity of $\ul$ and Lemma \ref{s-t-integral-inequality} serves for the proof.

To prove (\ref{om_gleich_0_bd_hat_D}): Let $y_n(t)=A_{1/n}(t)x_0,$ and $x_n(t)=J_{1/n}(t)x_0,$ then
$$
\nrm{y_n}\le \btr{A(t)x_0}\le f(t), \quad \nrm{J_{1/n}(t)x_0-x_0}\le f(t)\frac{1}{n}.
$$
For this sequence we have
$$ y_n(t) \in A(t)x_n(t).$$
Together with (\ref{int_half-line-Yosida-Approx} ), we obtain,
\begin{equation}
\semim{ \frac{1}{\la}\fk{ u_{\la}(t)-u_0 -\frac{1}{\la}\int_0^t\exp(-\frac{s}{\la})(u_{\la }(t-s)-u_0)ds}
-y_n(t),u_{\la}(t)-x_n(t)} \le 0.
\end{equation}
Consequently, there is $x_n^*(t) \in J(\ul(t)-x_n(t))$  s.t.
\begin{eqnarray} \label{dissipative_ineq_ul}
\lefteqn{\Rep x^*_n(t)\fk{\ul(t)-x_0-\frac{1}{\la}\int_0^t\exp(-\frac{s}{\la})(u_{\la }(t-s)-x_0)ds
	-\exp(-\frac{t}{\la})(x_0-u_0)-\la y_n(t)}} \nonumber \\
&=& \Rep x_n^*(t)\fk{ u_{\la}(t)-u_0-\frac{1}{\la}\int_0^t\exp(-\frac{s}{\la})(u_{\la }(t-s)-u_0)ds-y_n(t)} \nonumber \\
&\le& 0.
\end{eqnarray}
Applying $\semim{\cdot,\cdot}$ being lower semi continuous, we derive for $n\ge N,$ 
\begin{eqnarray*}
\lefteqn{\nrm{\ul(t)-x_0}}\\
&\le& \liminf_{n\to \infty} \semim{\ul(t)-x_0),\ul(t)-x_n(t)} \\
&\le& \liminf_{n\to \infty}\Rep x^*_n(t)\fk{\ul(t)-x_0)}\\
&\le&\liminf_{n\to\infty}\Big\{\Rep x^*_n(t)\fk{\ul(t)-x_0-\frac{1}{\la}\int_0^t\exp(-\frac{s}{\la})(u_{\la }(t-s)-x_0)ds
	-\exp(-\frac{t}{\la})(x_0-u_0)-\la y_n(t)}\\
&&+\Rep x_n^*(t)\fk{ u_{\la}(t)-u_0-\frac{1}{\la}\int_0^t\exp(-\frac{s}{\la})(u_{\la }(t-s)-u_0)ds-\la y_n(t)} \Big\}  \\
&\le& \frac{1}{\la}\intO\exp(-\frac{s}{\la})\nrm{u_{\la }(t-s)-x_0}ds +\exp(-\frac{t}{\la})\nrm{x_0-u_0}+\la f(t).
\end{eqnarray*}
Again we derive the inequality
\begin{eqnarray*}
\nrm{\ul(t)-x_0}&\le& \la f(t)+\el{t}\nrm{x_0-u_0} +\dl\int_0^t\el{s}\nrm{\ul(t-s)-x_0}ds
\end{eqnarray*}
a.e. $t\in\rep.$
An application of the continuity of $\ul$ and Lemma \ref{s-t-integral-inequality} serves for the proof.
\end{proof}

\begin{lem} \label{IVA0-IVE2-CauchyConvergence-lem}
Let $I=\rep,$  $u_0\in \hat{D},$ Assumption \ref{general-IVA0} and Assumption \ref{general-IVE2} be fulfilled, then
$\ul$ is locally uniformly Cauchy in $\la \to 0+.$
\end{lem}

\begin{proof}
 To prove $\ul$ to be Cauchy, we apply
\cite[Prop. 6.5 Ineq.(6.9), p. 187]{Ito_Kappel} with $t_1=t,t_2=s$,
\begin{eqnarray*}
x_1 &:=& u_{\la}(t) \\
x_2 &:=& u_{\mu}(s)\\
y_1 &:=& \frac{1}{\la}\fk{ u_{\la}(t)-u_0-\frac{1}{\la}\int_0^t \exp(-\frac{r}{\la})(u_{\la}(t-r)-u_0)dr} \\
y_2&:=&\frac{1}{\mu}\fk{ u_{\mu}(s)-u_0-\frac{1}{\mu}\int_0^s \exp(-\frac{r}{\mu})(u_{\mu}(s-r)-u_0)dr},
\end{eqnarray*}
the Lipschitz continuity of the control functions $h,g$, and Lemma \ref{IVE2-half-line-boundedness-lemma} to arrive at
\begin{eqnarray*}
\nrm{\ul(t)-\um(s)}
&\le&\frac{\la}{\lmw}\int_0^s\resm{\sigma}\nrm{\ul(t)-\um(s-\sigma)}d\sigma \\
&&+\frac{\mu}{\lmw}\int_0^s\res{\tau}\nrm{\ul(t-\tau)-\um(s)}d\tau \\
&&+\frac{\la}{\lmw}\nresm{s}\nrm{\ul(t)-u_0} \\
&&+\frac{\mu}{\lmw}\nres{t}\nrm{\um(s)-u_0} \\
&&+K\frac{\la\mu}{\lmw}\btr{t-s}.
\end{eqnarray*}

To bring  Lemmas \ref{general-2-dim-inequality} and \ref{general-2-dim-resolvent} into play, we choose the following setting:
\begin{gather*}
"t",\quad \alpha:=-\frac{1}{\la},\quad \delta:=\frac{\mu}{\la(\lmw)},\mbox{ consequently } \alpha+\delta=\frac{\mu\om -1}{\lmw},\\
"s",\quad \beta:=-\frac{1}{\mu},\quad \gamma:=\frac{\la}{\mu(\lmw)},\mbox{ consequently } \beta+\gamma=\frac{\la\om -1}{\lmw}.
\end{gather*}
The right hand side of the two dimensional integral inequality is
\begin{eqnarray*}
G(t,s)
&=&\frac{\la}{\lmw}\nresm{s}\nrm{\ul(t)-u_0} \\
&&+\frac{\mu}{\lmw}\nres{t}\nrm{\um(s)-u_0} \\
&&+K\frac{\la\mu}{\lmw}\btr{t-s}.
\end{eqnarray*}

Thus, we have to consider the following solution of the two dimensional integral equation. Define
$\Lambda:=\la+\mu-\la\mu\om.$ Then for $F(\la,\mu,t,s):=\nrm{\ul(t)-\um(s)},$ we have:
\begin{eqnarray}
\lefteqn{F(\la,\mu,t,s)}\\
&=&G(t,s) \\
&&+ \frac{\la}{\mu\Lambda}\int_0^s \exp(\llmw y)G(t,s-y)dy  \label{dint1}\\
&&+ \frac{\mu}{\la\Lambda}\int_0^t \exp(\mlmw x)G(t-x,s)dx \label{dint2}\\
&&+\frac{\la}{\mu\Lambda}\int_0^t\int_0^s \partial_x \IO{xy}\exp\fk{ \mlmw x + \llmw y}G((t-x,s-y)dydx \label{dint3}\\
&&+\frac{\mu}{\la\Lambda}\int_0^t\int_0^s \partial_y \IO{xy}\exp\fk{ \mlmw x + \llmw y}G((t-x,s-y)dydx \label{dint4}\\
&&+\frac{2}{\Lambda^2}\int_0^t\int_0^s \IO{xy}\exp\fk{ \mlmw x + \llmw y}G((t-x,s-y)dydx. \label{dint5} 
\end{eqnarray}
Similar to \cite{Gripenberg} the proof of uniform convergence for the $\ul$ is split into the equicontinuity, and the convergence in $L^1[0,T].$ We consider
$$\int_0^TF(\la,\mu,t,t)dt$$

We start with an estimation of the initial value terms.
Due to the symmetry in $(\la,t)$ and $(\mu,s)$ of $G(t,s)$ and the kernel, we only have to consider one of them.
$$
g(x,y):=\frac{\mu}{\lmw}\nres{x}\nrm{\um(y)-u_0}.
$$
The main step is to estimate (\ref{dint5}). Let
$$f(x,y):=\IO{xy}\exp\fk{ \mlmw x + \llmw y}.
$$ 
We go to apply Proposition \ref{interchange-integrals} and Lemma \ref{IVE2-half-line-boundedness-lemma} (1): 
\begin{eqnarray*}
\lefteqn{\int_0^T\int_0^t\int_0^t f(x,y)g(t-x,t-y)dydxdt} \\
&=& \int_0^T\int_x^T\int_{\max(x,y)-x}^{T-x} f(x,y)g(u,u+x-y)dudydx \\
&\le&\frac{\mu}{\lmw}\int_0^T\int_x^Tf(x,y)\int_{\max(x,y)-x}^{T-x}\nres{u}\nrm{\um(u+x-y)-u_0}dudydx \\
&\le& \frac{K\mu}{\lmw}\int_0^T\int_x^Tf(x,y)\int_{\max(x,y)-x}^{T-x}\nres{u}(u+x-y-\mu)dudydx \\
&\le& \frac{K\mu}{\lmw}\int_0^T\int_0^Tf(x,y)(\la^2+\la\btr{x-y}+\la\mu)dydx.
\end{eqnarray*}

This leads to the following integrals for the initial value term
$g(x,y):$
\begin{enumerate}
\item $$\frac{\mu\la}{\Lambda^3}\int_0^T\int_0^T \IO{xy}\exp\fk{ \mlmw x + \llmw y}\btr{x-y}dydx ,$$
\item $$\frac{\mu\la^2}{\Lambda^3}\int_0^T\int_0^T \IO{xy}\exp\fk{ \mlmw x + \llmw y}dydx ,$$
\item $$\frac{\mu^2\la}{\Lambda^3}\int_0^t\int_0^T \IO{xy}\exp\fk{ \mlmw x + \llmw y}dydy.$$
\end{enumerate}
Additionally, $G$ contains the symmetric ones in $(\la,\mu,x,y)$ for the second initial value term.
For the last term of $G$ we have
$$
\frac{\mu\la}{\Lambda^3}\int_0^T\int_0^T \IO{xy}\exp\fk{ \mlmw x + \llmw y}\btr{x-y}dydx,
$$
which is part of the estimate for $g.$ As $\la/\Lambda$ is bounded, for $K>0,$ 
$$\exp(\la x/\Lambda ) \le \exp(K T).$$

The previous observation and the substitution $x=x/\Lambda, y=y/\Lambda,$ lead to the following reduced integrals, which have to be discussed:
\begin{enumerate}
\item $$
\frac{ \mu\la }{\Lambda}\int_0^{T/\Lambda}\int_0^{T/\Lambda} I_0(\sqrt{xy})\exp(x -y)\btr{\Lambda x-\Lambda y}dydx, $$
\item $$\frac{\mu\la^2}{\Lambda}\int_0^{T/\Lambda}\int_0^{T/\Lambda} I_0(\sqrt{xy})\exp(-x-y)dydx, $$
\item $$\frac{\mu^2\la}{\Lambda}\int_0^{T/\Lambda}\int_0^{T/\Lambda} I_0(\sqrt{xy})\exp(-x-y)dydy.$$
\end{enumerate}
Defining $R=T/ \Lambda ,$ and  using $\la/\Lambda, \mu/\Lambda$ bounded, we obtain for the first integral
$$\frac{1}{R^2}\int_0^{R}\int_0^{R} I_0(\sqrt{xy})\exp(x -y)\btr{ x-y}dydx. $$
Applying Lemma \ref{SomeIntegrals} (\ref{bd_lipschitz_bessel_convergence}) we are finished.
The integrals (\ref{bessel_laplace_trans}) and (\ref{infinite_lipschitz_bessel_convergence}) will be proved by the same arguments, thanks to Lemma \ref{SomeIntegrals} (\ref{bd_bessel_convergence}). \\
Considering the integrals with the derivative of the Bessel-function $I_0$, we end up with:
\begin{enumerate}
\item $$
\mu^2\Lambda \int_0^{T/\Lambda}\int_0^{T/\Lambda} \partial_yI_0(\sqrt{xy})\exp(x -y)\btr{\Lambda x-\Lambda y}dydx, $$
\item $$\mu^2 \la \Lambda \int_0^{T/\Lambda}\int_0^{T/\Lambda} \partial_yI_0(\sqrt{xy})\exp(-x-y)dydx, $$
\item $$\mu^3  \Lambda \int_0^{T/\Lambda}\int_0^{T/\Lambda} \partial_yI_0(\sqrt{xy})\exp(-x-y)dydy. $$
\end{enumerate}
The above integrals tend to zero using Partial Integration and Lemma \ref{SomeIntegrals}.
The symmetry applies for the other ones.

The representative terms with exponential function as kernel are:
\begin{enumerate}
\item $$\exp(-\frac{t}{\mu})\frac{\mu}{\Lambda^2}\int_0^t\exp\fk{\frac{\mu\om-1}{\Lambda}x}\nrm{u_{\la}(t-x)-x_0}dx,$$
\item $$\frac{\mu^2}{\la\Lambda^2}\int_0^t\exp\fk{\frac{\mu\om-1}{\Lambda}x}\exp(\frac{x-t}{\la})\nrm{u_{\mu}(t)-x_0}dx,$$
\label{kernel_exp_2}
\item $$\frac{\mu^2}{\Lambda^2}\int_0^t\exp\fk{\frac{\mu\om-1}{\Lambda}y}xdx. $$
\end{enumerate}
Due to the local boundedness of $\ul,$  the terms could be easily computed and tend to zero as $\la,\mu \to 0.$ For (\ref{kernel_exp_2}) use, after integration, $x\exp(-x)$ bounded for $x>0.$
\end{proof}

\begin{cor} \label{u_of_t_in_D_hat}
Let $I=\rep,$ $u_0\in \hat{D},$ Assumption \ref{general-IVA0} and Assumption \ref{general-IVE2} be fulfilled, then the limit
$$
\lim_{\la\to 0}\ul(t)=u(t)\in\hat{D}.
$$
Moreover, there is $K>0,$ such that $\btr{A(t)u(t)}\le K$ for all $t\in[0,T].$
\end{cor}
\begin{proof}
We are going to apply \cite[Lemma 1.2, (ii),(iv), Lemma 1.4]{CrandallPazy}. As for $\mu$ small, 
$$
\nrm{A_{\mu}(t)\ul(t)} \le \frac{1}{1-\mu\om}\inf_{y\in A(t)\ul(t)}\nrm{y}\le \frac{1}{1-\mu\om} \nrm{\fk{\partial_t}_{\la}\ul(t)}\le K,
$$
by Lemma \ref{IVE2-half-line-boundedness-lemma},
we have
\begin{eqnarray*}
\nrm{A_{\mu}(t)u(t)}&\le& \nrm{A_{\mu}(t)\ul(t)}+\nrm{A_{\mu}(t)\ul(t)-A_{\mu}(t)u(t)}\\
&\le& K +L_{\mu}\nrm{\ul(t)-u(t)}.
\end{eqnarray*}
Letting $\la \lra 0+,$ and then $\mu\lra 0+,$ we obtain the claim.
\end{proof}

\begin{cor} \label{IVA0-IVE1-CauchyConvergence-cor}
Let $I=\rep,$ $u_0\in \hat{D},$ Assumption \ref{general-IVA0} and Assumption \ref{general-IVE1} be fulfilled, then
$\ul$ is locally uniformly Cauchy as $\la \to 0+.$
\end{cor}
\begin{proof}
 Applying Corollary \ref{IVE1-half-line-boundeness-cor}, we obtain, similar to Lemma \ref{IVA0-IVE2-CauchyConvergence-lem}, the inequality

\begin{eqnarray*}
\nrm{\ul(t)-\um(s)}
&\le&\frac{\la}{\lmw}\int_0^s\resm{\sigma}\nrm{\ul(t)-\um(s-\sigma)}d\sigma \\
&&+\frac{\mu}{\lmw}\int_0^s\res{\tau}\nrm{\ul(t-\tau)-\um(s)}d\tau \\
&&+\frac{\la}{\lmw}\nresm{s}\nrm{\um(t)-u_0} \\
&&+\frac{\mu}{\lmw}\nres{t}\nrm{\um(s)-u_0} \\
&&+K\frac{\la\mu}{\lmw}\nrm{h(t)-h(s)}.
\end{eqnarray*}
The initial value terms $\frac{\la}{\lmw}\nresm{s}\nrm{\um(t)-u_0},$ and $\frac{\mu}{\lmw}\nres{t}\nrm{\um(s)-u_0},$ are the same as in the proof of Lemma \ref{IVA0-IVE2-CauchyConvergence-lem}. Thus, it remains to discuss the control function
$$
\frac{\mu\la}{\Lambda^3}\int_0^T\int_0^T \IO{xy}\exp\fk{ \mlmw x + \llmw y}\btr{h(t-x)-h(t-y)}dydx.
$$
As
\begin{eqnarray*}
\lefteqn{\frac{\mu\la}{\Lambda^3}\int_0^T\int_0^T \IO{xy}\exp\fk{ \mlmw x + \llmw y}dydx}\\
&=&\frac{ \mu\la }{\Lambda}\int_0^{T/\Lambda}\int_0^{T/\Lambda} I_0(\sqrt{xy})\exp(x -y)dydx,
\end{eqnarray*}
is uniformly bounded by Lemma \ref{SomeIntegrals}, we can approximate $h$ with the mollified $h_{\ep}.$ Using $h_{\ep}$ Lipschitz, we end up with the integral,
$$
\frac{\mu\la}{\Lambda^3}\int_0^T\int_0^T \IO{xy}\exp\fk{ \mlmw x + \llmw y}\btr{x-y}dydx,
$$ 
which was discussed in the proof of Lemma \ref{IVA0-IVE2-CauchyConvergence-lem}.\\
\end{proof}

Next, we prove the existence and convergence for the whole line case before we show that the limit is an integral solution.

\section{Proofs of Technical Preliminaries in the Whole Line Case}
In this section we discuss the proofs for the whole line case.

Due to the uniform continuity of $h$ in the case of Assumption \ref{general-IVE1}, and of $g,h$ in the case of Assumption \ref{general-IVE2} for
 given $\la >0,$ $\Jl(t)f(t) \in \bucrx$ for all $f\in \bucrx.$ Thus, the solution $u_{\la}\in \bucrx$
   to equation (\ref{int_line-Yosida-Approx}) is found
 by applying the Banach Fix-Point Principle, where the mapping is
\begin{equation} \label{fixpoint-mapping}
F(u)(t):= J_{ \frac{\la}{1-\la \om} }(t) \left( \frac{1}{1-\la \om}
    \left(\intO \res{s}u(t-s)ds\right) \right).
\end{equation}

For this $\ul$ prove the uniform boundedness in $\la >0,$ and $t\in\re$ cited in Proposition \ref{ul-bounded} \\

A sufficient condition for $\ul$ to be bounded is:

\begin{pro} \label{ul-bounded} 
Let $I=\re,$  Assumption \ref{general-IVA0}, and either Assumption \ref{general-IVE1}, or Assumption \ref{general-IVE2} be fulfilled and $\om < 0.$ 

Then 
$$
\nrm{u_{\la}(t)} < K_1,
$$
for all $\la >0,$ and $t\in \re.$
\end{pro}

\begin{proof}
Let $x_0\in\hat{D},$ then thanks to Remark \ref{hat_D_constant}, $\btr{A(t)x_0}\le K_1,$ for all $t\in \re.$
Choose $y_n(t):=A_{1/n}(t)x_0,$ and $x_n:=J_{1/n}(t)x_0.$ For this choice we have,
$$
\nrm{y_n}\le \btr{A(t)x_0}\le K_1, \quad \nrm{J_{1/n}(t)x_0-x_0}\le K_1\frac{1}{n}.
$$
For this sequence we have,
$$ y_n(t) \in A(t)x_n(t).$$
Together with (\ref{int_line-Yosida-Approx}), we obtain,
\begin{equation}
\semim{ \frac{1}{\la}\fk{ u_{\la}(t) -\frac{1}{\la}\intO\exp(-\frac{s}{\la})u_{\la }(t-s)ds}
-\om u_{\la}(t)-y_n(t),u_{\la}(t)-x_n(t)} \le 0.
\end{equation}
Consequently there is $x_n^*(t) \in J(\ul(t)-x_n(t))$ such that
\begin{equation} \label{dissipative_ineq_ul1}
\Rep x_n^*(t)\fk{\frac{1}{\la}\fk{ u_{\la}(t)-\frac{1}{\la}\intO\exp(-\frac{s}{\la})u_{\la }(t-s)ds}-\om u_{\la}(t)-y_n(t)} \le 0.
\end{equation}
Noting $(\frac{d}{dt})_{\la}(\bk{t\mapsto x_0})=0,$ and applying $\semim{\cdot,\cdot}$ being lower semi continuous, we derive 
\begin{eqnarray*}
\lefteqn{(1-\la\om)\nrm{\ul(t)-x_0}}\\
&\le& \liminf_{n\to\infty} \semim{\fk{(1-\la\om)(\ul(t)-x_0)},\ul(t)-x_n(t)}  \\
&\le& \liminf_{n\to\infty}\Rep x^*_n(t)\fk{(1-\la\om)(\ul(t)-x_0)}\\
&\le& \liminf_{n\to\infty} \big\{ \Rep x^*_n(t)\fk{\ul(t)-\frac{1}{\la}\intO\exp(-\frac{s}{\la})u_{\la }(t-s)ds
	-\la\om \ul(t)-\la y_n(t)}\Big\}\\
&&+\frac{1}{\la}\intO\exp(-\frac{s}{\la})\nrm{u_{\la }(t-s)-x_0}ds +\la\om\nrm{x_0}+\la\nrm{y_n(t)} \\
&\le& \frac{1}{\la}\intO\exp(-\frac{s}{\la})\nrm{u_{\la }(t-s)-x_0}ds +\la\om\nrm{x_0}+\la K_1.
\end{eqnarray*}
Let $K_2:=\om\nrm{x_0}+K_1.$
An application of Proposition \ref{integral-ineq.} yields
\begin{eqnarray*}
\lefteqn{\nrm{u_{\la}(t)-x_0}} \\
&\le& \frac{\la}{1-\la \om}K_2
+\frac{1}{1-\la \om} \int_0^{\infty}\exp(\frac{\om \tau}{1-\la \om})K_2 d\tau \\
&\le& K_2\fk{\frac{\la}{1-\la \om} - \frac{1}{\om}}
\end{eqnarray*}
Thus, we obtain $\nrm{u_{\la}(t)}$ is uniformly bounded in $t\in \re$ and $\la >0.$
\end{proof}

\begin{remk}  \label{only_bd_cont_control_ul_bd}
Note that, for the boundedness of $\ul,$ only the boundedness of $g,h$ was used.
\end{remk}

\begin{lem} \label{uniform-bd-equi}
Let $I=\re,$ Assumption \ref{general-IVA0} be fulfilled with $\om <0.$ 
Then the family $\bk{\ul:\la >0}$ is  equicontinuous on $I$ in the case of  Assumption \ref{general-IVE1}. \\
If the Lipschitz constant of $g$ is less than $-\om,$  then in the case of Assumption \ref{general-IVE2},
there exists $K_3>0$ such that
$$
\nrm{\dl\fk{ \ul(t)-\dl\intO\el{s}\ul(t-s)ds} } \le K_3
$$
for all $\la >0,$ and $t\in \re.$ Further, the family $\bk{\ul:\la >0}$ has a common Lipschitz constant.
\end{lem}

\begin{proof}
For given $\la >0$ we choose in Assumption (\ref{general-IVE2}) $t_1=t, t_2=s,$
\begin{eqnarray*}
x_1&=&u_{\la}(t-s), \\
y_1&=& \frac{1}{\la}\left( u_{\la}(t-s)-\frac{1}{\la}\intO \exp(-\frac{r}{\la})u_{\la}(t-s-r)dr\right),\\
x_2&=&u_{\la}(t), \\
y_2&=& \frac{1}{\la}\left( u_{\la}(t)-\frac{1}{\la}\intO \exp(-\frac{r}{\la})u_{\la}(t-r)dr\right).
\end{eqnarray*}
With $K_2=\nri{\ul}$ this leads to
\begin{eqnarray} \label{yi_inequality}
\lefteqn{(1-\la \om) \nrm{u_{\la}(t)-u_{\la}(t-s)}} \nonumber \\
& \le& \intO\res{r}\nrm{u_{\la}(t-r)-u_{\la}(t-s-r)} \nonumber \\
&&+ \la \nrm{\hw(t-s)-\hw(t)} \Lw(K_2) + \la \nrm{g(t-s)-g(t)}\nrm{y_2}.
\end{eqnarray}
In the case Assumption \ref{general-IVE1}, i.e. $g=0,$ Prop. \ref{integral-ineq.} leads to
\begin{eqnarray}\label{assu22_equicontinuity_ineq}
\nrm{\ul(t)-\ul(t-s)} &\le& \la \nrm{h(t-s)-h(t)} L(K_2) \nonumber \\
&&+\int_0^{\infty}\exp(\frac{\om r}{1-\la\om})\nrm{h(t-s-r)-h(t-r)}L(K_2)dr, 
\end{eqnarray}
which proves the equicontinuity of $\bk{\ul:\la >0}.$ Thus, we may consider the case $g\not=0.$

Noting that $\nrm{y_2} \le \frac{2}{\la}\nri{u_{\la}}$ for a fixed $\la >0,$ we obtain, together with Proposition \ref{integral-ineq.} ,
\begin{eqnarray*}
\lefteqn{\nrm{\ul(t)-\ul(t-s)}\le \la \fk{ L_h\btr{s}\Lw(K_2) + \frac{2K_2}{\la}L_g\btr{s}}}\\
&&+\int_0^{\infty}\exp(\frac{\om r}{1-\la\om})\fk{\nrm{\hw(t-r)-\hw(t-r-s)}\Lw(K_2)+  \frac{2K_2}{\la} \nrm{g(t-r)-g(t-r-s)}} dr. \\
\end{eqnarray*}

Thus $u_{\la}$ is Lipschitz.
Let 
$$ K(t):= \sup_{s\ge0}\frac{1}{s}\nrm{u_{\la}(t)-u_{\la}(t-s)}.$$
Then we have,
\begin{eqnarray*}
\nrm{\frac{1}{\la}\left( u_{\la}(t)-\frac{1}{\la}\intO \exp(-\frac{r}{\la})u_{\la}(t-r)dr\right) } 
&\le&\frac{1}{\la^2}\intO\exp(-\frac{r}{\la})\nrm{u_{\la}(t)-u_{\la}(t-r)}dr ,\\
&\le& K(t) \frac{1}{\la^2}\intO \exp(-\frac{r}{\la}) r dr ,\\
&\le& K(t).
\end{eqnarray*}
Dividing inequality (\ref{yi_inequality}) by $s,$ we obtain with $L_h$ the Lipschitz constant for $\hw,$
\begin{eqnarray*}
(1-\la \om) K(t) &\le& \intO \res{r} K(t-r)dr \\
&&+ \la \sup_{s\ge0}L_h \Lw(\nrm{u_{\la}(s)}) \\
&&+ \la \sup_{s\ge0}\frac{1}{s}\nrm{g(t)-g(t-s)}K(t).
\end{eqnarray*}
Rearranging, and letting $L_g$ the Lipschitz constant of $g,$ gives
\begin{eqnarray*}
K(t)&\le& \frac{1}{1-\la(\om+L_g)}\intO\res{r}K(t-r)dr \\
&&+\la \sup_{s\ge0}\frac{1}{s}\nrm{\hw(t)-\hw(t-s)}\Lw(\nrm{u_{\la}(s)}).
\end{eqnarray*}
Now, Proposition \ref{integral-ineq.} serves for the proof of boundedness of $K(t),$ and consequently for the equicontinuity of $\ul$.
\end{proof}

\begin{cor}\label{IVA0-IVE1-whole-line-equicont}
Let $I=\re,$ Assumption \ref{general-IVA0} be fulfilled and $\om <0.$ Further, let Assumption \ref{general-IVE1} be fulfilled with $h\in C_b(\re,X).$ Then the family $\bk{\ul:\la >0}$ is locally  equicontinuous on $\re.$ 
\end{cor}
\begin{proof}
 Due to Remark \ref{only_bd_cont_control_ul_bd}, $\ul$ is uniformly bounded, hence we can apply the inequality (\ref{assu22_equicontinuity_ineq}).
\end{proof}

\begin{lem} \label{the_ineq.}
Let $I=\re,$ Assumption \ref{general-IVA0}, and either Assumption \ref{general-IVE1} or Assumption \ref{general-IVE2} be fulfilled ,$\om <0$ and $u_{\la}$ and $u_{\mu}$ be approximants to equation (\ref{int_line-Yosida-Approx}).
Then
\begin{eqnarray*}
\lefteqn{(\la+\mu-\la \mu \om)\nrm{u_{\la}(t)-u_{\mu}(s)} \le \la \intO\resm{r}\nrm{u_{\la}(t)-u_{\mu}(s-r)}dr} \\
&&+\mu \intO\res{r}\nrm{u_{\la}(t-r)-u_{\mu}(s)} dr \\
&& + \la \mu \nrm{\hw(t)-\hw(s)}\Lw(\nrm{u_{\mu}(s)}) \\
&&+\la \mu \nrm{g(t)-g(s)}\nrm{\frac{1}{\mu}\fk{ u_{\mu}(s)-\frac{1}{\mu}\intO \exp(-\frac{r}{\mu})u_{\mu}(s-r)dr}}.
\end{eqnarray*}
\end{lem}

\begin{proof}
Apply \cite[Prop. 6.5 Ineq.(6.9), p. 187]{Ito_Kappel} with $t_1=t,t_2=s$,
\begin{eqnarray*}
x_1 &:=& u_{\la}(t), \\
x_2 &:=& u_{\mu}(s),\\
y_1 &:=& \frac{1}{\la}\fk{ u_{\la}(t)-\frac{1}{\la}\intO \exp(-\frac{r}{\la})u_{\la}(t-r)dr}, \\
y_2&:=&\frac{1}{\mu}\fk{ u_{\mu}(s)-\frac{1}{\mu}\intO \exp(-\frac{r}{\mu})u_{\mu}(s-r)dr}.
\end{eqnarray*}
Rearranging, and the identity $\frac{1}{\mu}\intO \exp(-\frac{r}{\mu})=1,$ serve for the proof.
\end{proof}

\begin{lem} \label{Cauchy-solutions-lem}

Let $I=\re,$ then we have:
\begin{enumerate}
\item If Assumption \ref{general-IVA0} and Assumption \ref{general-IVE1} with $\om < 0$ are fulfilled, then the Yosida approximants $\fk{u_{\la}:\la >0}$ to equation (\ref{int_line-Yosida-Approx}) are Cauchy in $BUC(\re,X,$ when $\la \to 0+.$
\item If Assumption \ref{general-IVA0} and Assumption \ref{general-IVE2} are fullfilled, and the Lipschitz constant of $g$ in Assumption \ref{general-IVE2} is less than $-\om,$ then the Yosida approximants $\fk{u_{\la}:\la >0}$ to equation (\ref{int_line-Yosida-Approx}) are Cauchy in \bucrx when $\la \to 0+ $.
\end{enumerate}
\end{lem}

\begin{proof}
 We start with the case Assumption \ref{general-IVA0} and Assumption \ref{general-IVE2}. 
 Using the inequality of Lemma \ref{the_ineq.}, and Lemma \ref{uniform-bd-equi} gives,
\begin{eqnarray*}
\lefteqn{(\la+\mu-\la \mu \om)\nrm{u_{\la}(t)-u_{\mu}(s)}}\\
&\le& \la \intO\resm{r}\nrm{u_{\la}(t)-u_{\mu}(s-r)}dr
+\mu \intO\res{r}\nrm{u_{\la}(t-r)-u_{\mu}(s)} dr \\
&& + \la \mu \nrm{\hw(t)-\hw(s)}\Lw(K_2)+\la \mu \nrm{g(t)-g(s)}K_3. 
\end{eqnarray*} 
The resolvent is given by \cite[Lemma 4.2, p.648]{Kreulichbd}.
Estimating via Lipschitz continuity the $u$ in the cited paper to be considered becomes
$$ 
u(t,s)=\la \mu K\btr{t-s}.
$$
Due to the boundedness of $\hw,g$ and \cite[Remark 4.1 p.647]{Kreulichbd}, the main term in the resolvent after substituting $x:=x/\Lambda, y:=y/\Lambda$ is
$$
\la\mu\int_0^{\infty}\int_0^{\infty}I_0(\sqrt{xy})\exp(-(1-\mu\om)x-(1-\la\om)y)\btr{x-y}dxdy.
$$
Consequently, Lemma \ref{SomeIntegrals} (\ref{infinite_lipschitz_bessel_convergence}) serves for the proof.
\end{proof}

The Assumption \ref{general-IVE1} requires a reapplication of the arguments, together with an approximation of the control function with Lipschitz $h_{\ep}.$ 
By uniform continuity the mollified $h_{\ep} \to h$  uniformly on  $\re.$ Noting $h_{\ep}$ Lipschitz, we can apply \cite[Lemma 4.2 (1)]{Kreulichbd} and Lemma \ref{SomeIntegrals} (\ref{infinite_lipschitz_bessel_convergence}) to obtain
$$
\frac{\mu\la}{\Lambda^3}\int_0^{\infty}\int_0^{\infty} \IO{xy}\exp\fk{ \mlmw x + \llmw y}\btr{h(t-x)-h(t-y)}dydx \to 0
$$
for $\la,\mu \to 0$. \\

\begin{lem} \label{cpt_open_convergence}
Let $I=\re.$ If in the case of Assumption \ref{general-IVA0}, and Assumption \ref{general-IVE1},  where the control function $h\in C_b(\re,X),$ i.e. only bounded continuous, $\ul$ is locally uniformly convergent.
\end{lem}
\begin{proof}
 Using $h_{\ep}\to h$ uniformly on bounded sets, Lemma \ref{SomeIntegrals} (\ref{bessel_cpt_open_convergence}) with
$$g_{\ep}(t,s):=\btr{h_{\ep}(t)-h_{\ep}(s)} \lra g(t,s):=\btr{h(t)-h(s)},$$ 
the proof finishes, using the local Lipschitz continuity of $h_{\ep},$ and Lemma \ref{SomeIntegrals} (\ref{bd_lipschitz_bessel_convergence}).
\end{proof}

\section{Proofs of the Main Results}

After the technical preliminaries we are in the situation to prove that the solution found as limits of the Yosida approximation are integral ones.

\begin{proof}[Proof of Theorem \ref{Cauchy-solutions}]
From Lemma \ref{Cauchy-solutions-lem} we learned that the Yosida approximations in the given cases are uniformly convergent on $\re.$ We will start with the proof of item (2).
For given $t,r \in \re,$ with $t>r$ we have:
\begin{eqnarray*}
\lefteqn{\dl\fk{\ul(t)-\dl\int_0^{t-r}\el{\tau}\ul(t-\tau)d\tau}} \\
&&+\dl \el{t-r}\ul(r) 
+ \frac{1}{\la^2} \int_{t-r}^{\infty} \el{\tau} \fk{\ul(t-\tau)-\ul(r)}d\tau \\
&\in& A(t)\ul(t)+\om\ul(t).
\end{eqnarray*}
As
$$ \dl\fk{x-\dl\int_0^{t-s}\el{\tau}xd\tau}+\dl\el{t-s}x=0,$$
we have
\begin{eqnarray*}
\dl\fk{x-\dl\int_0^{t-s}\el{\tau}xd\tau}+\dl\el{t-s}x &\in& A(r)x+\om x -y,
\end{eqnarray*}
for a $y\in A(r)x+\om x.$
Applying \cite[Prop. 6.5]{Ito_Kappel}, we obtain in the case of Assumption \ref{general-IVE2},

\begin{eqnarray*}
\lefteqn{[\dl \fk{ \ul(t)-\dl \int_0^{t-r}\el{\tau}\ul(t-\tau)d\tau } +\dl \el{t-r} \ul(r)} \\
&&+\frac{1}{\la^2} \int_{t-r}^{\infty} \el{\tau} \fk{\ul(t-\tau)-\ul(r)}d\tau \\
&&-\dl\fk{x-\dl\int_0^{t-s}\el{\tau}x d\tau}-\dl\el{t-s}x -y,\ul(t)-x]_{-} \\
&\le& \om\nrm{\ul(t)-x}+ \nrm{\hw(t)-\hw(r)}\Lw(\nrm{x})+ \nrm{g(t)-g(r)}\nrm{y}.
\end{eqnarray*}

Multiplying by $\la,$ rearranging, and the substitution $\tau=t-\tau$ lead to
\begin{eqnarray*}
\nrm{\ul(t)-x}&\le& \dl \int_r^t \el{t-\tau}\nrm{\ul(\tau)-x}d\tau +\la \el{t-r}\nrm{\ul(r)-x} \\
&&+\dl \int_{t-r}^{\infty} \el{\tau} \fk{\ul(t-\tau)-\ul(r)}d\tau + \la [y,\ul(t)-x]_{+} \\
&&+\la \om\nrm{\ul(t)-x}+ \la \nrm{\hw(t)-\hw(r)}\Lw(\nrm{x})+ \la \nrm{g(t)-g(r)}\nrm{y}.
\end{eqnarray*}
Thus, we are in the situation to apply Proposition \ref{ineq-finite-integral}. Hence

\begin{eqnarray}
\nrm{\ul(t)-x}&\le& \la \el{t-r}\nrm{\ul(r)-x} +\la \om\nrm{\ul(t)-x} \nonumber\\
&&+\dl \int_{t-r}^{\infty} \el{\tau} \nrm{\ul(t-\tau)-\ul(r)}d\tau \label{1} \\ 
&& + \la [y,\ul(t)-x]_{+} \nonumber \\
&&+ \la \nrm{\hw(t)-\hw(r)}\Lw(\nrm{x})+ \la \nrm{g(t)-g(r)}\nrm{y} \nonumber \\
&&+ \int_r^t\fk{\om\nrm{\ul(\tau)-x} +[y,\ul(\tau)-x]_{+}}d\tau \nonumber \\
&&+\int_r^t\fk{\nrm{\hw(\tau)-\hw(r)}\Lw(\nrm{x})+ \nrm{g(\tau)-g(r)}\nrm{y}}d\tau \nonumber \\
&&+\dl\int_r^t\el{\tau-r}d\tau\nrm{\ul(r)-x} \label{2} \\
&&+\frac{1}{\la^2}\int_r^t\int_0^{\infty}\el{r-\mu-\nu}\nrm{\ul(r-\mu)-\ul(r)}d\mu d\nu .\label{3}
\end{eqnarray}
As $\la \to 0+,$ we have to consider the terms (\ref{1}), (\ref{2}), and (\ref{3}).

Term (\ref{1}): 
\begin{eqnarray*}
\dl \int_{t-r}^{\infty} \el{\tau} \nrm{\ul(t-\tau)-\ul(r)}d\tau &\le& 2\el{t-r}\sup_{t\in\re}\nrm{\ul(t)} ,\\
&& \lra 0.
\end{eqnarray*}
Term (\ref{2}):
\begin{eqnarray*}
\dl\int_r^t\el{\tau-r}d\tau\nrm{\ul(r)-x} &=& (1-\el{t-r})\nrm{\ul(r)-x}, \\
&& \lra \nrm{\ul(r)-x},
\end{eqnarray*}
when $\la\to 0+.$ \newline
Term(\ref{3}): Due to Lemma \ref{uniform-bd-equi} we have
equicontinuity on $\re$ of  $\ul$ for $\la >0.$
Thus, for given $\ep > 0,$ there exist  $ \delta > 0 ,$ such that, whenever $\btr{r-s}<\delta,$ we have
$$\nrm{\ul(r)-\ul(s)}<\ep, \ \ \forall \la >0 .$$
In consequence, we have:
\begin{eqnarray*}
\lefteqn{\frac{1}{\la^2}\int_r^t\int_0^{\infty}\el{r-\mu-\nu}\nrm{\ul(r-\mu)-\ul(r)}d\mu d\nu} \\
&\le& \frac{\ep}{\la^2}\int_r^t\int_0^{\delta}\el{r-\mu-\nu} d\mu d\nu  \\
&&+\frac{1}{\la^2}\int_r^t\int_{\delta}^{\infty}\el{r-\mu-\nu}2\nri{\ul}d\mu d\nu ,\\
&\le&\ep (1-\el{r-t})(1-\el{\delta}) \\
&&+2(1-\el{r-t})\el{\delta}\nri{\ul}.
\end{eqnarray*}
Thus, as $\la \to 0+,$ we obtain
$$
\limsup_{\la\to 0+}\frac{1}{\la^2}\int_r^t\int_0^{\infty}\el{r-\mu-\nu}\nrm{\ul(r-\mu)-\ul(r)}d\mu d\nu =0.
$$
Passing to the limit we obtain the desired inequality. \\
To prove item (1) note that only boundedness and local equicontinuity was used, what is proved in Propostion \ref{ul-bounded} and Corollary \ref{IVA0-IVE1-whole-line-equicont}.
\end{proof}

\begin{proof}[Proof of Theorem \ref{IVA0-IVE1-CauchyConvergence} and \ref{IVA0-IVE2-CauchyConvergence}]

For given $t,r \in \rep,$ with $t>r,$ $\ul$ the Yosida approximation, and from Lemma \ref{IVA0-IVE2-CauchyConvergence-lem} and Corollary \ref{IVA0-IVE1-CauchyConvergence-cor} we learned that $\ul$ is locally uniformly Cauchy. Further, we have by the definition of the approximation
\begin{eqnarray*}
\lefteqn{\dl\fk{\ul(t)-\dl\int_0^{t-r}\el{\tau}\ul(t-\tau)d\tau} +\dl \el{t-r}\ul(r) }\\
&&+ \frac{1}{\la^2} \int_{t-r}^{t} \el{\tau} \fk{\ul(t-\tau)-\ul(r)}d\tau \\
&&+\frac{1}{\la}\el{t}(u_0-\ul(r)) \\
&\in& A(t)\ul(t)+\om\ul(t).
\end{eqnarray*}
Now, the same method as in the whole line case applies.
\end{proof}

\section{Proofs: Inhomgeneous Problem}

For the proof of Theorem \ref{buc_int_sol}, we need some preliminary results.
The inhomogeneous case will be viewed as a special case of homogeneous one. We define
$$ B(t):=\bk{[x,y+f(t)]:[x,y] \in A(t)}.$$
Then equation (\ref{whole-line-inhomogeneous-equation}) is equivalent to
\begin{equation} \label{whole-line-B-inhomogeneous-equation}
u^{\prime}(t)\in B(t)u(t)+\om u(t), t \in \re,
\end{equation}

and (\ref{half-line-inhomogeneous-equation}) to
\begin{equation} \label{half-line-B-inhomogeneous-equation}
u^{\prime}(t)\in B(t)u(t)+\om u(t), u(0)=u_0, t \in \rep.
\end{equation}

To apply the previous results we have to prove that $B(t)$ satisfies Assumption \ref{general-IVE1}, or Assumption \ref{general-IVE2} if $A(t)$ does.

\begin{pro} \label{B(t)+f(t)E2)}
\begin{enumerate}
\item Let $A(t)$ fulfill Assumptions \ref{general-IVA0}, \ref{general-IVE1}, and let $f\in BUC(\rep,X).$ Then $B(t)$ fulfills Assumptions \ref{general-IVA0} and \ref{general-IVE1}.
\item Let $A(t)$ fulfill Assumptions \ref{general-IVA0}, \ref{general-IVE2}, and let $f$ be bounded and Lipschitz. Then $B(t)$ fulfills Assumptions \ref{general-IVA0} and \ref{general-IVE2}.
\end{enumerate}
\end{pro}
\begin{proof}[Proof of Proposition \ref{B(t)+f(t)E2)}]
We prove the Proposition in the case of Assumption \ref{general-IVA0} and Assumption \ref{general-IVE2}. Let $t_1,t_2 \in \re,$ and $[x_i,y_i]\in A(t_i);$ then Assumption \ref{general-IVE2} implies
\begin{eqnarray*}
\lefteqn{(1-\la \om)\nrm{x_1-x_2}} \\
&\le& \nrm{x_1-x_2 -\la(y_1-y_2)}+\la \nrm{\hw(t_1)-\hw(t_2)}\Lw(\nrm{x_2}) \\
&&+ \la \nrm{g(t_1)-g(t_2)}\nrm{y_2} \\
&\le& \nrm{x_1-x_2 -\la(y_1+f(t_1)-y_2-f(t_2))} + \la\nrm{f(t_1)-f(t_2)} \\
&&+\la \nrm{\hw(t_1)-\hw(t_2)}\Lw(\nrm{x_2})+ \la \nrm{g(t_1)-g(t_2)}\nrm{y_2} \\
\end{eqnarray*}
Thus,
$$
\la\nrm{f(t_1)-f(t_2)} +\la \nrm{\hw(t_1)-\hw(t_2)}\Lw(\nrm{x_2})+ \la \nrm{g(t_1)-g(t_2)}\nrm{y_2}
$$
is an estimate, which allows us to redo the existence proof, where only Lipschitz continuity and boundedness was needed.
\end{proof}

\begin{proof}[Proof of Proposition \ref{half-line-inhomgeneous-stability}] A simple computation gives 
\begin{equation} \label{Resolvent-identity}
\Jl^{B(t)}(t)y=\Jl^{A(t)}(t)(\la f(t)+y).
\end{equation}
This identity leads  to
$$
u_{i,\la}(t)=\Jl^A{(t)}\fk{\la f_{i}(t)+\exp(-t/\la)x_{i} +
\frac{1}{\la}\int_0^t\exp(-\tau/\la)u_{i,\la}(t-\tau)d\tau},
$$
for $i=1,2.$ Thus,
\begin{eqnarray*}
\lefteqn{\nrm{u_{1,\la}(t)-u_{2,\la}(t)}}\\
&\le&\frac{1}{1-\la\om}\fk{\la\nrm{f_1(t)-f_2(t)}+\exp(-t/\la)\nrm{x_1-x_2}} \\
&&+\frac{1}{\la(1-\la\om)}\int_0^t\exp(-\tau/\la)\nrm{u_{1,\la}(t-\tau)-u_{2,\la}(t-\tau)}d\tau.
\end{eqnarray*}
Hence, an application of Lemma \ref{s-t-integral-inequality} leads to

\begin{eqnarray}\label{half-line-u-la-stability}
\lefteqn{\nrm{u_{1,\la}(t)-u_{2,\la}(t)}} \nonumber \\
&\le&\frac{1}{1-\la\om}\fk{\la\nrm{f_1(t)-f_2(t)}+\exp(-t/\la)\nrm{x_1-x_2}} \nonumber \\
&&+ \exp\fk{\frac{\om}{1-\la\om}t}\nrm{x_1-x_2}  
+\frac{1}{1-\la\om}\int_0^t\exp\fk{\frac{\om(t-\tau)}{1-\la\om}}\nrm{f_1(\tau)-f_2(\tau)}d\tau.
\end{eqnarray}
Now, depending on the assumptions \ref{general-IVE1}, or \ref{general-IVE2}, Corollary \ref{IVA0-IVE1-CauchyConvergence-cor}, or Lemma \ref{IVA0-IVE2-CauchyConvergence-lem} apply for the proof.
\end{proof}

\begin{cor}\label{half-line-u-la-stability-cor}
Let $I=\rep,$ and  $A(t)$ fulfill Assumption \ref{general-IVA0}. Further let $L_g<-\om,$ $L_g$ the Lipschitz constant of $g,$ and  Assumption \ref{general-IVE2} be fulfilled. For given right hand sides  $f_1, f_2\in\bucrpx,$ $x_1,x_2\in X,$ the corresponding Yosida approximants $u_{1,\la},u_{2,\la}$ on $\re$ fulfill inequality (\ref{half-line-u-la-stability}).
\end{cor}

\begin{remk}
Due to the previous stability we may consider initial values in the closure of the domain, i.e. in $\overline{D}.$
\end{remk}

\begin{proof}[Proof of Proposition \ref{inhomgeneous-stability}]
Due to (\ref{Resolvent-identity}) the fixpoint identity is given by,

$$
u_{i,\la }(t)=F(u_{i,\la})(t)= J_{ \frac{\la}{1-\la \om } }^{A(t)}(t) 
\fk{ \frac{1}{1-\la \om} \fk {\la f_i(t)+\intO \res{s}u_{ i,\la }(t-s)ds} }.
$$
Thus, using $\Jl^{A(t)}$ being a contraction,
\begin{eqnarray*}
\nrm{ u_{1,\la}(t)-u_{2,\la}(t)}&\le& \frac{\la}{1-\la \om}\nrm{f_1(t)-f_2(t)} \\
&&+\frac{\la}{1-\la \om}\intO \res{r}\nrm{u_{1,\la}(t-r)-u_{2,\la}(t-r)}dr .
\end{eqnarray*}
Applying Proposition. \ref{integral-ineq.},
\begin{eqnarray} \label{whole-line-u-la-stability}
\nrm{u_{1,\la}(t)-u_{2,\la}(t)}&\le& \frac{\la}{1-\la \om}\nrm{f_1(t)-f_2(t)} \nonumber \\
&&+\fk{\frac{1}{1-\la \om}}^2\intO\exp\fk{ \frac{\om r}{1-\la\om} }\nrm{f_1(t-r)-f_2(t-r)}dr.
\end{eqnarray}
The conclusion follows when $\la \to 0+,$ as we have the uniform convergence of the $\ul$ in both cases.
\end{proof}

\begin{proof}[Proof of Theorem \ref{uniform-whole line-convergence}]
Consider for given $f\in\bucrx$ the mollified $f_{\ep}.$ For $\la,\mu >0$ the pairs of solutions and right hand sides $(\ul,f)(\um,f),(u_{\la}^{\ep},f_{\ep}),$ and $(u_{\mu}^{\ep},f_{ep})$ of equation (\ref{line-Yosida-Approx}). Now the inequality (\ref{whole-line-u-la-stability}, and Theorem \ref{Cauchy-solutions} with $A(t):=A(t)+f_{\ep}(t)$ serve for the proof.
\end{proof}

We are now ready for the
\begin{proof}[Proof of Theorem \ref{buc_int_sol} (\ref{buc_int_sol_re}):] We give only the whole line case proof.
Let $f_{\ep}$ be the mollified $f,$ then $f_{\ep}$ is bounded and Lipschitz, and consequently,
\begin{equation*}
u^{\prime}_{\ep}(t)\in A(t)u_{\ep}(t)+\om u_{\ep}(t)+f_{\ep}(t), t\in \re,
\end{equation*}
has an integral solution $u_{\ep}.$ Redoing the steps of  the proof of Theorem \ref{Cauchy-solutions}, for given $t,r \in \re,$ with $t>r,$ we have
\begin{eqnarray*}
\lefteqn{\dl\fk{\ue(t)-\dl\int_0^{t-r}\el{\tau}\ue(t-\tau)d\tau}} \\
&&+\dl \el{t-r}\ue(r) \\
&&+ \frac{1}{\la^2} \int_{t-r}^{\infty} \el{\tau} \fk{\ue(t-\tau)-\ul(r)}d\tau - f_{\ep}(t) \\
&\in& A(t)\ue(t)+\om\ue(t) ,
\end{eqnarray*}
and 
\begin{eqnarray*}
\dl\fk{x-\dl\int_0^{t-s}\el{\tau}xd\tau}+\dl\el{t-s}x +y &\in& A(r)x+\om x.
\end{eqnarray*}
The previous choice, a multiplication by $\la,$ rearranging and the substitution $\tau=t-\tau$ leads to
\begin{eqnarray*}
\nrm{\ue(t)-x}&\le& \dl \int_r^t \el{t-\tau}\nrm{\ue(\tau)-x}d\tau +\la \el{t-r}\nrm{\ue(r)-x} \\
&&+\dl \int_{t-r}^{\infty} \el{\tau} \fk{\ue(t-\tau)-\ul(r)}d\tau + \la [y+f(t),\ue(t)-x]_{+} \\
&&+\la \om\nrm{\ue(t)-x}+ \la \nrm{\hw(t)-\hw(r)}\Lw(\nrm{x})+ \la \nrm{g(t)-g(r)}\nrm{y}.
\end{eqnarray*}
Now the proof becomes similar to the one of Theorem \ref{Cauchy-solutions}. Apply Proposition \ref{ineq-finite-integral}, and pass to $\la \to 0+.$ We obtain that $u_{\ep}$ is an integral solution in the sense of Definition 
\ref{def-int-sol-inhomo}.
The stability result Proposition \ref{inhomgeneous-stability} gives that $u_{\ep}$ is Cauchy in \bucrx.
Thus, we may pass to limits in the inequality of the integral solution on $\re.$\\
\end{proof}

\begin{proof}[Proof of Theorem \ref{buc_int_sol} (\ref{buc_int_sol_rep}):] Reapply the previous method
\end{proof}

\begin{proof}[Proof of Corollary \ref{compare-half-whole-line}] Use that $u$ is the unique integral solution with initial value $u(0).$
\end{proof}

\begin{proof}[Proof of Theorem \ref{uniform_convergence_half_line}]
 Without loss of generality $f\in \bucrx.$ Due to Corollary \ref{half-line-u-la-stability-cor} the fixpoints for $f\in \bucrx,$ and its mollified $f_{\ep}$ are uniformly close. Hence we may assume $f$ to be Lipschitz.
Rewrite for $t\in \rep$ equation (\ref{int_line-Yosida-Approx})  with the solution $\vl$ as,
\begin{eqnarray*}
\lefteqn{\frac{1}{\la}\fk{\vl(t)-\vl(0)-\frac{1}{\la}\int_0^t\exp\fk{\frac{\tau-t}{\la}}(\vl(\tau)-\vl(0))d\tau}}\\
&\in&A(t)\vl(t)+\om\vl(t)+f(t)+\frac{1}{\la^2}\int_{-\infty}^0\exp\fk{\frac{\tau-t}{\la}}(\vl(\tau)-\vl(0))d\tau,
\end{eqnarray*}
and its comparison with the approximated initial value equation (\ref{half-line-B-inhomogeneous-equation}):
\begin{eqnarray*}
\lefteqn{\frac{1}{\la}\fk{\ul(t)-u_0-\frac{1}{\la}\int_0^t\exp\fk{\frac{\tau-t}{\la}}(\ul(\tau)-u_0)d\tau}}\\
&\in&A(t)\ul(t)+\om\ul(t)+f(t),
\end{eqnarray*}
shows that the proof of \cite[Lemma 2.3]{Kreulichbd} applies using Corollary \ref{half-line-u-la-stability-cor}.
\end{proof}

\begin{proof}[Proof of Theorem \ref{solution-in-Y}]
 Using the fixpoint mapping (\ref{fixpoint-mapping}) we have that the integral solution is the limit of functions in $Y.$
 \end{proof}

\begin{proof}[Proof of Theorem \ref{solution-in-APRX}]
 To apply Theorem \ref{solution-in-Y}, it only has to be proved that $\bk{t\mapsto \Jl(t)f(t)}\in \apx$ for all $f\in \apx.$
But that is a consequence of \cite[Chapter VII,Lemma 4.1]{Dal-Krein}.
\end{proof}

\begin{proof}[Proof of Proposition \ref{solution-in-AARX}] As $AA(\re,X)\subset C_b(\re,X),$ due to Lemma \ref{cpt_open_convergence}, we find a bounded solution $u$ to (\ref{int_whole-Line-Eq}), and the approximants converge locally uniformly. The approximants satisfy the inequality (\ref{assu22_equicontinuity_ineq}),
\begin{eqnarray}
\nrm{\ul(t)-\ul(t-s)} &\le& \la \nrm{h(t-s)-h(t)} L(K_2) \nonumber \\
&&+\int_0^{\infty}\exp(\frac{\om r}{1-\la\om})\nrm{h(t-s-r)-h(t-r)}L(K_2)dr.
\end{eqnarray}
The assumption $h\in AA(\re,X)$ leads, for a given sequence in $\re,$ to a subsequence $\bk{s_k}_{k\in\za},$ such that
$$
\lim_{k\to\infty}\lim_{l\to\infty}h(t+s_k-s_l)=h(t)
$$
pointwise for every $t\in\re.$
Passing $\la\to 0+,$  we have
\begin{eqnarray}
\nrm{u(t+s_l)-u(t+s_k)} &\le& 
\int_0^{\infty}\exp(\om r)\nrm{h(t+s_l-r)-h(t+s_k-r)}L(K_2)dr.
\end{eqnarray}
Consequently,    $\bk{u(t+s_l)}_{l\in\za}$ is a Cauchy sequence for every $t\in\re.$ Thus, we can choose
$t:=t-s_l,$ which leads to
\begin{eqnarray}
\nrm{u(t)-u(t+s_k-s_l)} &\le& 
\int_0^{\infty}\exp(\om r)\nrm{h(t+s_k-s_l-r)-h(t-r)}L(K_2)dr.
\end{eqnarray}
Thus, Lebesque's Convergence Theorem applies.
\end{proof}

\section{Application to Functional-Differential-Equations}

We consider the following type of nonautonomous nonlinear Functional-Differential-Equations
\begin{equation} \label{Funtional-Equation}
u^{\prime}(t)\in A(t)u(t)+\om u(t)+G(t,u_t), t\in \re.
\end{equation} 
Here, $E:=(C([-r,0],X),\nrE{\cdot}),$ or $ E:=(BUC((-\infty,0],X),\nrE{\cdot}),$ and
$$
\Funk{G}{\re \times E}{X}{(t,\phi)}{G(t,\phi)}
$$
is such that for a constant $\beta >0,$
$$
\nrm{G(t,\phi_1)-G(t,\phi_2)}\le \beta\nrm{\phi_1-\phi_2}_E
$$
for all $\phi_1,\phi_2  \in E.$
\begin{defi}
A function $u \in \bucrx$ is called the integral solution to (\ref{Funtional-Equation}), if for $f(t):=G(t,u_t),$ $u$ is the integral solution to (\ref{whole-Line-inh-Eq}) in the sense of Definition \ref{def-int-sol-inhomo}.
\end{defi}

\begin{theo} \label{buc_functional_sol}
Let $Y \subset \bucrx$ be a closed and translation invariant subspace, $A(t)$ fulfill Assumption \ref{general-IVA0} and either Assumption \ref{general-IVE1} or Assumption \ref{general-IVE2}, with  $L_g$ the Lipschitz constant of $g,$ $L_g<-\om,$ and  
$$\bk{t\mapsto J_{\la}(t)u(t)},\bk{t\mapsto G(t,u_t)} \in Y,$$ for all $u\in Y.$ Then, if also $\beta < -\omega,$ equation (\ref{Funtional-Equation}) has an integral solution $u \in Y$  on $\re.$
\end{theo} 
\begin{proof}  Let $T:Y \lra Y$ denote the mapping with respect to Theorem \ref{buc_int_sol}, and Theorem \ref{solution-in-Y} which  maps, for given $v \in Y,$ $f:=\bk{t\mapsto G(t,v_t)}$ on the solution $u\in Y,$ which is an integral solution. Then an application of Proposition \ref{inhomgeneous-stability} leads to
\begin{eqnarray*}
\nrm{Tv_1(t)-Tv_2(t)}&\le& \int_0^{\infty}\exp(\om s)\nrm{G(r,(v_1)_{t-s})-G(r,(v_2)_{t-s})}ds \\
&\le& \beta\int_0^{\infty}\exp(\om s)\sup_{\nu \in [-r,0]}\nrm{v_1(t-s+\nu)+v_2(t-s+\nu)}ds \\
&\le&\frac{\beta}{-\om}\sup_{r\in \re}\nrm{v_1(r)-v_2(r)}
\end{eqnarray*}
Consequently, $T$ is a contraction on $Y$, and the fixpoint $u$ is the uniform limit of integral solutions in $Y,$ therefore itself an integral solution in $Y.$ 
\end{proof}

\begin{pro} Let $u:\re \to X$ be the integral solution of (\ref{Funtional-Equation}), and $v:I\cup\rep \lra X$ the mild solution to 
\begin{eqnarray*} 
v(t)&\in& A(t)v(t)+\om v(t)+G(t,v_t) \ \ t\ge 0 \\
v_{|I}&=&u_{|I}. 
\end{eqnarray*}

Then $u(t)=v(t)$ for all $t\ge 0.$
\end{pro}

\begin{proof} As $u$ is an integral solution with the right hand side $G(t,u_t),$ and $v$ is a mild solution, using \cite[Theorem 6.30 p.232]{Ito_Kappel}, we obtain,
\begin{eqnarray*}
\exp(-\om t)\nrm{u(t)-v(t)}&\le&\int_0^t\exp(-\om \nu)\nrm{G(s,u_{\nu})-G(s,v_{\nu})}d\nu \\
&\le&\beta\int_0^t\exp(-\om \nu)\nrE{u_{\nu}-v_{\nu}}d\nu.
\end{eqnarray*}
Due to the monotonicity of the integral, we have
\begin{eqnarray*}
\exp(-\om s)\nrm{u(s)-v(s)}&\le&\beta\int_0^t\exp(-\om \nu)\nrE{u_{\nu}-v_{\nu}}d\nu
\end{eqnarray*}
for all $0\le s \le t.$ Consequently, in the finite delay case for $t\ge r$
\begin{eqnarray*}
\exp(-\om s)\nrE{u_s-v_s}&\le&\beta\int_0^t\exp(-\om s)\nrE{u_{\nu}-v_{\nu}}d\nu,
\end{eqnarray*}
the Gronwall Lemma gives $u_t=v_t$ for all $t\ge r,$ which completes the proof.
In the infinite delay case, use that $u(t)=v(t)$ for $t\le 0,$ and therefore
$\exp(-\om t) \nrm{u(t)-v(t)} \le 0.$ Again Gronwall's Lemma serves for the proof.
\end{proof}

To discuss the object
\begin{eqnarray} \label{FDE-half-line-eq}
u^{\prime}(t)&\in& A(t)u(t) +\om u(t) +G(t,u_t)  \\
u_{|I}&=& \vp \in E \nonumber
\end{eqnarray}
for the case of finite delay, we obtain:
\begin{pro}
Under the conditions of Theorem \ref{buc_functional_sol}, in the finite delay case with $L <-\om,$ every solution to (\ref{FDE-half-line-eq}) is asymptotically close to $Y.$
\end{pro}
\begin{proof} Apply Remark \ref{hat_D_constant} and  \cite[Prop. 4.3]{Ghavidel}. 
\end{proof}

Consequently, we are able to compare the integral solution on $\re$ with the mild solution in the case of exponential asymptotic stability for the FDE. Moreover, only 
$$
\lim_{t\to \infty}\nrm{u(t)-v(t)} =0,
$$
where $u$ denotes the solution on $\re,$ is needed. For exponential stability, we cite the following results of \cite[Th. 4.3. Cor. 4.2, and Prop. 4.3]{Ghavidel} and \cite{Ruess_Summers_stability}.

\section{Some Integrals}

The following gives a list of the properties needed:
\begin{enumerate}
  \item $\displaystyle I_0(2\sqrt{\alpha x y}) =
         \sum_{k=0}^{\infty} \frac{\alpha^kx^ky^k}{(k!)^2} =
           \frac{2}{\pi}\int_0^1\frac{1}{\sqrt{1-t^2}}
           \cosh(2t\sqrt{\alpha xy})dt$
  \item $I_0(0)=1,$ $\partial_x I_0(2\sqrt{\alpha x })_{|x=0}=0$
  
  \item $ \displaystyle
         %  \partial_1J(x,y) =
           \partial_x I_0(2\sqrt{\alpha xy})
            = \sqrt{\frac{\alpha y}{x}} I_1(2\sqrt{\alpha xy}) $
  \item$ \partial_x\partial_y  I_0(2\sqrt{\alpha x y}) =
          \alpha I_0(2\sqrt{\alpha x y}) $
 
\end{enumerate}
\begin{lem} \label{SomeIntegrals}
 Let $I_0$ be the modified Besselfunction, $\om,\la,\mu >0,$ and $\Lambda=\la+\mu+\la\mu\om.$ Then:
\begin{enumerate}
\item $$ \lim_{R\to\infty}\frac{1}{R^2}\int_0^R\int_0^R I_0(2\sqrt{xy})\exp(-x-y)\btr{x-y}dydx =0. $$ \label{bd_lipschitz_bessel_convergence}
\item $$ \int_0^{\infty}\int_0^{\infty} I_0(2\sqrt{xy})\exp(-(1+\la\om)x-(1+\mu\om)y)dydx=\frac{1}{\om\Lambda}.$$ 
\label{bessel_laplace_trans}
\item $$ \lim_{\la,\mu \to 0} \la\mu \int_{0}^{\infty} \int_{0}^{\infty} \btr{x-y}I_0(2\sqrt{xy})\exp(-x-y)\exp(-\om(\la x+\mu y)dydx=0.$$
\label{infinite_lipschitz_bessel_convergence}
\item $$\frac{1}{R}\int_0^R \int_0^R I_0(2\sqrt{xy})\exp(-x-y)dydx \le 1.$$
\label{bd_bessel}
\item 
$$ 
\lim_{R\to\infty}\frac{1}{R^2}\int_0^R \int_0^R I_0(2\sqrt{xy})\exp(-x-y)dydx =0 .
$$
\label{bd_bessel_convergence}
\item  $$ \lim_{R\to\infty}\frac{1}{R^2}\int_0^R \int_0^R \partial_yI_0(2\sqrt{xy})\exp(-x-y)\btr{x-y}dydx =0 .$$
\label{bd_bessel_1_convergence}
\item 
$$
\lim_{t\to\infty}\frac{\la\mu}{\Lambda^3}\int_{\rep^2\backslash [0,t]\times [0,t]} \IO{xy}\exp\fk{ \llmw x + \mlmw y}dydx=0
$$
uniformly in $\la,\mu >0.$
\label{bessel_uniform_integrability}
\item Let  $\bk{g_n}_{n\in\za}\subset C_b(\re\times\re,X),$ uniformly bounded, such that $g_n\lra 0$ unifomly on compact sets. Then
$$
\lim_{n\to\infty}\frac{\la\mu}{\Lambda^3}\int_0^{\infty}\int_0^{\infty} \IO{xy}\exp\fk{ \llmw x + \mlmw y}\nrm{g_n(x,y)}dydx=0
$$
uniformly in $\la,\mu >0.$ \label{bessel_cpt_open_convergence}
\end{enumerate}
\end{lem}

\begin{proof} Using the identity given in the NIST Handbook, http://dlmf.nist.gov p. 259(10.43.23), we have

\begin{equation}
\int_0^{\infty} t^{\nu+1} I_{\nu}(bt) \exp(-p^2t^2)dt= 
\frac{b^{\nu}}{(2p^2)^{\nu+1}}\exp\fk{\frac{b^2}{4p^2}}.
\end{equation}

We obtain for $\nu=0,x=t^2, 2tdt=dx$

\begin{equation}
\int_0^{\infty}I_0(b\sqrt{x})\exp(-p^2x)dx=\frac{1}{p^2}\exp\fk{\frac{b^2}{4p^2}}.
\end{equation}

Choosing $b:=2\sqrt{y},$ the previous equation gives

\begin{equation} \label{main_bessel_equation}
\int_0^{\infty}I_0(2\sqrt{xy})\exp(-p^2x)dx=\frac{1}{p^2}\exp\fk{\frac{y}{p^2}}.
\end{equation}

Therefore, (\ref{bd_bessel}) and (\ref{bd_bessel_convergence}) become a simple computation, i.e.,
$$
\lim_{R\to\infty}\frac{1}{R}\int_0^R\int_0^{\infty}I_0(2\sqrt{xy})\exp(-x-y)dxdy=1.
$$
Consequently,
$$
\lim_{R\to\infty}\frac{1}{R^2}\int_0^R\int_0^{\infty}I_0(2\sqrt{xy})\exp(-x-y)dxdy=0.
$$

{\bf Proof of (\ref{bd_lipschitz_bessel_convergence}):}
Note that the integrand is positive, and therefore the 
double-integral is increasing in $R>0$, and therefore either infinite or finite. As the finite case is obvious, assume the double-integral to be infinite. Hence we are in the case of the Bernoulli-l-H\^opital Rule. Applying

$$
\partial_t\int_a^t\int_a^tf(x,y)dxdy=\int_a^tf(t,y)dy+\int_a^tf(x,t)dx,
$$
and the symmetry of the integrand, we only have to prove
\begin{equation}
\lim_{R\to\infty}\frac{1}{R}\int_0^R (R-y)I_0(2\sqrt{Ry})\exp(-R-y)dy =0.
\end{equation}
We split the integration path $[0,R]=[0,R/2]\cup [R/2,R].$ 
Applying $I_0(x)\le \exp(x),$ for $x \ge 0,$  and $(R-y)\le R$, we obtain
\begin{eqnarray*}
\lefteqn{\frac{1}{R}\int_0^{R/2}(R-y)I_0(2\sqrt{Ry})\exp(-R-y)dy} \\
&\le& \frac{1}{R/2}\int_0^{R/2}R\exp\fk{-R\fk{1-\frac{1}{\sqrt{2}}}^2}dy,\\
&& \to 0
\end{eqnarray*}
as $R\to \infty.$ 

Thus, we can consider the asymptotic behavior of the integrand. Note that
\cite[p. 251,(8.08)]{Olver} gives
$$I_0(x) \thickapprox \frac{e^x}{\sqrt{2\pi x}}\fk{1+O\fk{\frac{1}{x}}}. $$
This together with
$(R-2\sqrt{Ry}+y)=\fk{\sqrt{R}-\sqrt{y}}^2,$
yields

\begin{eqnarray*}
\lefteqn{\frac{1}{R}\int_{R/2}^R(R-y)\frac{\exp\fk{-\fk{\sqrt{R}-\sqrt{y}}^2}}{\sqrt[4]{R}\sqrt[4]{y}}dy} \\
&\le&\frac{1}{R\sqrt{R}}\int_{R/2}^R (R-y)\exp\fk{-\fk{\sqrt{R}-\sqrt{y}}^2}dy \\
&=&\frac{1}{R\sqrt{R}}\int_{R/2}^R (\sqrt{R}-\sqrt{y})(\sqrt{R}+\sqrt{y})\exp\fk{-\fk{\sqrt{R}-\sqrt{y}}^2}dy. \\
\end{eqnarray*}
The subsitution $u=\sqrt{R}-\sqrt{y}$ transforms the integral into
\begin{eqnarray*}
\frac{1}{R\sqrt{R}} \int_0^{\sqrt{R}-\sqrt{R/2}}(2Ru-4\sqrt{R}u^2-u^3)\exp(-u^2)du &\to 0
\end{eqnarray*}
as $R\to \infty$. This finishes the proof of (\ref{bd_lipschitz_bessel_convergence}).\\

{\bf Proof of(\ref{infinite_lipschitz_bessel_convergence}):}
We define $R:=\frac{1}{\la+\mu}$ and split the region of integration to
\begin{eqnarray*}
\lefteqn{\bk{(x,y):0\le x\le R,0\le y\le R}}\\
&\cup & \bk{(x,y):0\le x\le R, R\le y <\infty}\\ 
&\cup & \bk{(x,y):R\le x<\infty, 0\le y<\infty}.
\end{eqnarray*}
Starting with the first region, we have
\begin{eqnarray*}
\lefteqn{\la\mu \int_{0}^{R} \int_{0}^{R} \btr{x-y}I_0(2\sqrt{xy})\exp(-x-y)\exp(-\om(\la x+\mu y)dydx} \\
&\le &\frac{\la\mu}{(\la+\mu)^2}\frac{1}{R^2}\int_{0}^{R} \int_{0}^{R} \btr{x-y}I_0(2\sqrt{xy})\exp(-x-y)dydx.
\end{eqnarray*}
Applying the previously proved result of Lemma \ref{SomeIntegrals} (\ref{bd_lipschitz_bessel_convergence}), it is sufficient to consider the regions
\begin{equation} \label{first-region-split}
\bk{(x,y):0\le x\le R, R\le y <\infty}\cup\bk{(x,y):R\le x <\infty, x \le y <\infty}
\end{equation}
 and 
$$
\bk{(x,y):0\le y\le R, R\le x <\infty}\cup\bk{(x,y):R\le y < \infty, y \le x <\infty}.
$$
 Due to the symmetry of the integrand and the domains it remains to show the claim for one of the regions. We choose (\ref{first-region-split}). We start with
$$
\bk{(x,y):R\le x <\infty, x \le y <\infty}.
$$
 
As we can use the asymptotics of the Besselfunction, $\la x\exp(-\la x)$ bounded for $x\ge 0,$ and the substitution $u=\sqrt{y}-\sqrt{x},$ we obtain the integral
\begin{eqnarray*}
\lefteqn{\la\mu\int_R^{\infty}\int_{x}^{\infty}\frac{(y-x)\exp\fk{-\fk{\sqrt{y}-\sqrt{x}}^2}\exp(-\om(\la x+\mu y))}{\sqrt[4]{x}\sqrt[4]{y}}dydx} \\
&\thickapprox&\int_R^{\infty}\int_{x}^{\infty}
\frac{(y-x)\exp\fk{-\fk{\sqrt{y}-\sqrt{x}}^2}}{x\sqrt[4]{x}y\sqrt[4]{y}}dydx \\
&\le&\int_R^{\infty}\int_{0}^{\infty}\frac{u(u+2\sqrt{x})(u+\sqrt{x})\exp(-u^2)}{(u+\sqrt{x})^2x\sqrt{u+\sqrt{x}}\sqrt[4]{x}} dudx.
\end{eqnarray*}
As $u+\sqrt{x} \ge \sqrt{x},$  we have $\sqrt{u+\sqrt{x}} \ge \sqrt[4]{x}.$ We end up with the integral
$$
\int_R^{\infty}\int_{0}^{\infty}\frac{(2ux+4\sqrt{x}u^2+u^3)\exp(-u^2)}{x^2\sqrt{x}}dudx \lra 0
 $$
 for $R\to \infty.$
It remains to consider
\begin{eqnarray*}
\lefteqn{\bk{(x,y):0\le x\le R, R\le y <\infty}} \\
&=& \bk{(x,y):0\le x\le \frac{R}{2}, R\le y <\infty}\cup \bk{(x,y):\frac{R}{2}\le x\le R, R\le y <\infty}\\
&=:& \Omega_1\cup \Omega_2.
\end{eqnarray*}
Using the previous observations, the discussion of $\Omega_2$ becomes straightforward. Again we can use the asymptotics of the Besselfunction, and the substitution $u=\sqrt{y}-\sqrt{x}.$
This leads to the integral
\begin{eqnarray*}
\lefteqn{\int_{R/2}^{R}\int_{\sqrt{R}-\sqrt{x}}^{\infty}\frac{(2ux+4\sqrt{x}u^2+u^3)\exp(-u^2)}{x^2\sqrt{x}}dudx}\\
&\le&\int_{R/2}^{R}\int_{0}^{\infty}\frac{(2ux+4\sqrt{x}u^2+u^3)\exp(-u^2)}{x^2\sqrt{x}}dudx \lra 0
\end{eqnarray*}
for $R\to \infty.$
 
Therefore it remains to consider $\Omega_1.$ Using the inequality $I_0(x)\le \exp(x),$ and finally the subtitution $u=\sqrt{y}-\sqrt{x},$ the integral tranforms into
$$
\la\mu\int_0^{R/2}\int_{\sqrt{R}-\sqrt{x}}^{\infty}u(u+\sqrt{x})(u+2\sqrt{x})\exp(-u^2)\exp(-\la\om x- \mu\om(u+\sqrt{x})^2)dudx.
$$
As $\mu\om(u+\sqrt{x})^2\exp(-\mu\om(u+\sqrt{x})^2)$ is bounded, for some $K>0,$ we have,
\begin{eqnarray*}
\lefteqn{\la\mu\int_0^{R/2}\int_{\sqrt{R}-\sqrt{x}}^{\infty}u(u+\sqrt{x})(u+2\sqrt{x})\exp(-u^2)\exp(-\la\om x- \mu\om(u+\sqrt{x})^2)dudx} \\
&\le&2 K\la\int_0^{R/2}\int_{\sqrt{R}-\sqrt{x}}^{\infty} u\exp(-u^2)\exp(-\la\om x)dudx \\
&\le& K \frac{1}{\om}\exp(-R(1-\frac{1}{\sqrt2})^2) \lra 0
\end{eqnarray*}
for $R\to \infty.$\\
{\bf Proof of (\ref{bd_bessel_1_convergence}):} By virtue of the partial integration formula for absolutely continuous functions, the proof becomes straightforward: 

\begin{eqnarray*}
\lefteqn{\frac{1}{R^2}\int_0^R \int_0^R \partial_y\left[I_0(2\sqrt{xy})\right]\exp(-x-y)\btr{x-y}dxdy}\\
&=&\frac{1}{R^2}\int_0^RI_0(2\sqrt{xy})\exp(-x-y)\btr{x-y}|_{y=0}^{y=R} dx \\
&&+\frac{1}{R^2}\int_0^R \int_0^R I_0(2\sqrt{xy})\partial_y[\exp(-x-y)\btr{x-y}]dxdy, \\
&\le&\frac{1}{R^2}\fk{\int_0^R I_0(2\sqrt{Rx})\exp(-R-x)Rdx - \int_0^RI_0(0)\exp(-x)xdx} \\
&&+\frac{1}{R^2}\int_0^R \int_0^R I_0(2\sqrt{xy})\exp(-x-y)\btr{x-y}dxdy \\
&&+\frac{1}{R^2}\int_0^R \int_0^R I_0(2\sqrt{xy})\exp(-x-y)\btr{\mbox{sign}(x-y)}dxdy. \\
\end{eqnarray*}
Thus, the previous integral results, and the equation (\ref{main_bessel_equation}) will apply for the proof.

{\bf Proof of (\ref{bessel_uniform_integrability}):} The forthcoming equations use the substitution $x:=x/\Lambda,y=y/\Lambda,$ and (\ref{main_bessel_equation}), with $p^2=1+\la\om.$

\begin{eqnarray*}
\lefteqn{\frac{\la\mu}{\Lambda^3} \int_t^{\infty}\int_0^{\infty}I_0\fk{2\frac{\sqrt{xy}}{\Lambda}}\exp\fk{-\frac{1+\la\om}{\Lambda}x-\frac{1+\mu\om}{\Lambda}y}dxdy}\\
&=&\frac{\la\mu}{\Lambda} \int_{t/\Lambda}^{\infty}\int_0^{\infty}I_0\fk{2\sqrt{xy}}\exp\fk{-(1+\la\om)x-(1+\mu\om)y}dxdy, \\
&=&\frac{\la\mu}{\Lambda}\int_{t/\Lambda}^{\infty}\exp(-(1+\mu\om)y)\frac{\exp\fk{\frac{y}{1+\la\om}}}{1+\la\om} dy \\
&=&\frac{\la\mu}{\Lambda}\int_{t/\Lambda}^{\infty}\exp\fk{-\frac{\om\Lambda}{1+\la\om}y}dy \\
&=&\frac{\la\mu}{\Lambda^2\om}\exp\fk{-\frac{\om t}{1+\la\om}} \\
&&\lra 0, \mbox{ when $t\to \infty,$ uniformly in $\la,\mu>0.$}
\end{eqnarray*}

{\bf Proof of (\ref{bessel_cpt_open_convergence}):}
Applying \cite[Lemma 4.2 (1)]{Kreulichbd} we have
$$
\frac{\la\mu}{\Lambda^3}\int_0^{\infty}\int_0^{\infty} \IO{xy}\exp\fk{ \llmw x + \mlmw y}dydx \le \frac{1}{\om},
$$
and due to the fact that every single term of the resolvent is positive, for the region $(0,t)\times(0,t),$ we obtain,
\begin{eqnarray*}
\frac{\la\mu}{\Lambda^3}\int_0^{t}\int_0^{t}\IO{xy}\exp\fk{ \llmw x + \mlmw y}\nrm{g_n(x,y)}dydx
&\le & \frac{1}{\om}\sup_{0\le x,y\le t}\nrm{g_n(x,y)}.
\end{eqnarray*}
Consequently, the integral tends to zero when $n\to \infty$ uniformly in $\la,\mu >0.$
\end{proof}

\begin{pro} \label{interchange-integrals}
Let $f,g\in L^1([0,T]^2),$ be positive. Then
\begin{eqnarray*}
\lefteqn{\int_0^T\int_0^t\int_0^tf(x,y)g(t-x,t-y)dxdydt} \\
&\le&\int_0^T\int_x^T\int_{\max(x,y)-x}^{T-x}f(x,y)g(u,u+x-y)dudydx.
\end{eqnarray*}
\end{pro}
\begin{proof}
\begin{eqnarray*}
\lefteqn{\int_0^T\int_0^t\int_0^tf(x,y)g(t-x,t-y)dxdydt} \\
&=& \int_0^T\int_0^T\int_0^T\chi_{[0,t]}(x)\chi_{[0,t]}(y)f(x,y)g(t-x,t-y)dydxdt \\
&=&\int_0^T\int_0^T\int_0^T\chi_{[0,t]}(x)\chi_{[0,t]}(y)f(x,y)g(t-x,t-y)dydtdx \\
&\le&\int_0^T\int_x^T\int_0^T\chi_{[0,t]}(y)f(x,y)g(t-x,t-y)dydtdx \\
&\le&\int_0^T\int_x^T\int_{\max(x,y)}^T f(x,y)g(t-x,t-y)dtdydx
\end{eqnarray*}
Substitute in the inner integral $u:=t-x,$ i.e. $x=t-u,$ the integral transforms to
$$\int_0^T\int_x^T\int_{\max(x,y)-x}^{T-x}f(x,y)g(u,u+x-y)dudydx.$$
\end{proof}

\section{Appendix}

\begin{lem} 
Let $\bk{f_n}_{n\in \za}\subset C[0,T]$ equicontinuous and convergent in $L^1[0,T]$ to some $f$, then $f_n\to f $ uniformly on $[0,T].$
\end{lem}
\begin{proof} Given an arbitrary subsequence there is a subsequence such that
$f_{n_{k_l}}\to f$ pointwise on $[0,T]\backslash N,$  with $\mu(N)=0.$ As $[0,T]\backslash N$  is dense, by the equicontinuity we obtain that $f_{n_{k_l}}\to f$ everywhere. Thus $f_n\to f$ uniformly. 
\end{proof}

Let 
$$
\dual{y,x}_i=\nrm{x}\lim_{\al \downarrow 0}\frac{\nrm{x}-\nrm{x-\al y}}{\al},
$$
$J(x):=\bk{x^*:\nrm{x}=x^*(x), \nrm{x^*}=1},$ and $F(x):=\bk{x^*:\nrm{x^*}^2=\nrm{x}^2= x^*(x)}. $
Recall:
\begin{enumerate}
\item $\dual{\cdot,\cdot}_- $ is lower semicontinuous,
\item $\dual{y,x}_-=\inf\bk{\Rep (y,f):f\in J(x)}=\min\bk{\Rep (y,f):f\in J(x)}.$
\end{enumerate}

\begin{pro}\label{lambda_plus_mu_ineq}
Let $x_1,x_2,y_1,y_2 \in X,\om>0 ,$ and $f\in\re.$ Then
\begin{equation} \label{eq1}
\nrm{x_1-x_2}\le \nrm{x_1-x_2-\la(y_1-y_2)}+\la f
\end{equation}
for all $0<\la<\om,$ if and only if,
\begin{equation}\label{eq2}
\dual{y_1-y_2,x_1-x_2}_i\le \nrm{x_1-x_2}f.
\end{equation}
Further, if one of the previous inequalities holds then
\begin{equation}\label{eq_la_mu}
(\la+\mu)\nrm{x_1-x_2}\le \la\nrm{x_2-x_1-\mu y_2} + \mu\nrm{x_1-x_2-\la y_1} +\la\mu f
\end{equation}
for all $0<\la,\mu<\om.$
\end{pro}
\begin{proof} The proof is similar to the one in \cite[Prop. 6.5, pp. 187-188]{Ito_Kappel}.  From the inequality (\ref{eq1}) we obtain from the definition of $\dual{\cdot,\cdot}_i$ ,
$$
\dual{y_1-y_2,x_1-x_2}_i\le \nrm{x_1-x_2}f.
$$
Let the inequality (\ref{eq2}) hold, we choose $x^*\in F(x_1-x_2),$ such that $\dual{y_1-y_2,x_1-x_2}_i=\Rep x^*(y_1-y_2).$
Then
\begin{eqnarray*}
\nrm{x_1-x_2}^2&=&\Rep x^*(x_1-x_2)=\Rep x^*(x_1-x_2-\la(y_1-y_2))+\la \Rep x^*(y_1-y_2) \\
&\le&\nrm{x_1-x_2}\fk{\nrm{x_1-x_2-\la(y_1-y_2)}+\la f}.
\end{eqnarray*}
Thus, the first part of the Lemma is proved.

To prove the inequality (\ref{eq_la_mu}), for $\la,\mu >0,$ and $x^*\in F(x_1-x_2),$ we obtain
\begin{eqnarray*}
(\la+\mu)\nrm{x_1-x_2}^2 &=& \la x^*(x_1-x_2)+\mu x^*(x_1-x_2), \\
&=& \mu \Rep x^*(x_1-x_2-\la y_1) -\la\Rep x^*(x_2-x_1-\mu y_2) +\la\mu \Rep x^*(y_1-y_2), \\
&\le& \nrm{x_1-x_2}\fk{\mu\nrm{x_1-x_2-\la y_1}+\la\nrm{x_2-x_1-\mu y_2} + \la\mu f}.
\end{eqnarray*}
Hence the claim is proved. 
\end{proof}

\begin{remk} \label{Integral-Power-Estimate}
Let $k\in \ctq$ and 
$$
\Funk{L}{\ctq}{\ctq}{f}{\bk{(t,s)\mapsto \int_0^tk(t,\tau)f(\tau,s)d\tau}}
$$
then 
$$\nri{L^nf}\le \frac{T^n\nri{k}^n\nri{f}}{n!}
$$
\end{remk}

\begin{remk} \label{Sum-quasinilpotent}
Let $X$ be a Banach space, and $L,S:X \to X$  linear  quasinilpotent and commuting operators
Then $L+S$ is quasinilpotent, i.e. the spectralradius of $L+S$ is zero.
\end{remk}

\begin{proof}
To prove the Remark use \cite[Theorem 11.23 ]{RudinFA}, i.e. let $x,y\in A,$ $A$ a Banach algebra (possibly noncommutative), and $xy=yx,$
then $\sigma(x+y)\subset \sigma(x)+\sigma(y).$ 
\end{proof}

\begin{pro} \label{bd-limsup-Lipschitz}
Let $f\in C((0,T),X),$ s.t.
$$ \limsup_{s \to 0+} \frac{1}{s}\nrm{f(t+s)-f(t)}\le L\quad \forall t\in(0,T).$$
Then $f$ is Lipschitz with a Lipschitz-constant less than or equal to $L.$ 
In particular, if $f$ is Lipschitz we have for the Lipschitz constant $L$
$$L\le \sup_{t\in (0,T)}\limsup_{s\to 0+}\frac{1}{s}\nrm{f(t+s)-f(t)}$$
\end{pro}

\begin{remk} Due to the uniformity on $(0,T),$ the same is true when 
$\frac{1}{s}\nrm{f(t-s)-f(t)}$ is considered, substitute $t:=t-s.$
\end{remk}

\begin{pro} \label{integral-ineq.}
The solution to the integral equation 
\begin{equation} \label{infinite-integral-inequality}
u(t) = f(t) + \alpha\intO\exp(- \beta \tau) u(t-\tau) d\tau, 
\end{equation}
 for 
$ 0 < \alpha < \beta,  $ 
is given by 
$$ u(t) = (Rf)(t):= f(t) + \alpha \intO\exp(-(\beta-\alpha)\tau) f(t-\tau)d\tau. $$ 
Note that the resolvent is positive.
\end{pro}
\begin{proof}
Considering
$$
\Funk{S}{BUC(\re)}{BUC(\re)}{u}{\bk{t \mapsto \alpha \int_0^{\infty}\exp(-\beta \tau)u(t-\tau)d\tau},}
$$
we have $\nri{S} \le \frac{\alpha}{\beta} < 1, $ and consequently
$$ (I-S)^{-1}=\sum_{j=0}^{\infty}S^j.$$
Thus the positivity of $S$ carries over to $(I-S)^{-1}.$ To compute the resolvent, multiply equation (\ref{infinite-integral-inequality})
with $\exp(\beta t),$ and define,
$$
F(t):=\int_{-\infty}^t \exp(\beta \tau)u(\tau)d\tau.
$$
Then $F$ has to be the solution to the differential equation,
$$
 F^{\prime}(t)=\exp(\beta t)f(t)+\alpha F(t). $$
Thus,
$$
F(t)=\exp(\alpha(t-s))F(s)+\int_s^t\exp(\beta(t-\tau))\exp(\beta\tau)f(\tau)d\tau.
$$
For $s \to -\infty$ we have,
\begin{eqnarray*}
\btr{\exp(\alpha(t-s))F(s)}&=&\btr{\exp(\alpha(t-s))\int_{-\infty}^s\exp(\beta \tau)u(\tau)d\tau} \\
&\le&\nrm{u}\exp(\alpha t)\exp((\beta-\alpha)s) \\
&& \to 0
\end{eqnarray*}
In consequence we obtain
$$
F(t)=\int_{-\infty}^t\exp(\alpha(t-\tau))\exp(\beta\tau)f(\tau)d\tau.
$$
Differentiation and multiplication with $\exp(-\beta t)$ completes the proof.
\end{proof}

\begin{lem} \label{s-t-integral-inequality} \label{ineq-finite-integral}
Let $t>a$, $\alpha>0$ and $u,f \in C[0,T]$ such that
$$
u(t)\le f(t)+\alpha \int_a^t\exp(-\beta(t-\tau))u(\tau)d\tau.
$$
Then
$$
u(t)\le f(t)+\alpha\int_a^t\exp((\alpha-\beta)(t-\tau))f(\tau)d\tau.
$$
\end{lem}

\begin{proof}
For $a\le t\le b,$ consider the operator
$$
\Funk{T}{C[a,b]}{C[a,b]}{u}{\bk{t\mapsto \alpha \int_a^t\exp(-\beta(t-\tau))u(\tau)d\tau}.}
$$
Clearly, $T$ is quasi-nilpotent. Thus
$$ (I-T)^{-1}=\sum_{j=0}^{\infty}T^j,$$
and the positivity of $T$ carries over to $(I-T)^{-1}.$
The representation of the resolvent is a similar computation to the proof of Proposition \ref{integral-ineq.}. The substitute $F(t):=\int_a^t\exp(\alpha \tau)u(\tau)d\tau$ serves for the proof.\\

The method of doubling the variables leads to an integral inequality of the following type:
\begin{equation} \label{2-dim-inequality}
Y(t,s)-\delta\int_0^t\exp(\alpha(t-\tau))Y(\tau,s)d\tau -\gamma\int_0^s\exp(\beta(s-\sigma))Y(t,\sigma)d\sigma \le G(t,s)
\end{equation}
Due to the previous Remarks we obtain that for $\delta,\gamma >0$ the equation has a positive resolvent.
The aim of this section is the positivity and to compute the resolvent. Thus, we have to consider the following equation,
\begin{equation} \label{2-dim-equation}
F(t,s)-\delta\int_0^t\exp(\alpha(t-\tau))F(\tau,s)d\tau -\gamma\int_0^s\exp(\beta(s-\sigma))F(t,\sigma)d\sigma = G(t,s).
\end{equation}
\end{proof}

\begin{lem} \label{general-2-dim-inequality}
With $\delta,\gamma >0,$ let $Y,G \in \ctq$ and let $Y$ fulfill the inequality (\ref{2-dim-inequality}). Then for the solution $F\in \ctq$ of the equation (\ref{2-dim-equation}), the inequality
$$Y(t,s)\le F(t,s)$$
holds for all $(t,s)\in [0,T]^2.$
\end{lem}
\begin{proof}
Due to the previous Remark \ref{Integral-Power-Estimate} and Proposition \ref{Sum-quasinilpotent}, we have that 
the resolvent to (\ref{2-dim-equation}) exists and is positive due to the representation as a Neumann series.\\
For the representation of the resolvent, the modified Bessel-Functions $I_0,$ and $I_1$ come into play.
\end{proof}

\begin{lem}\label{general-2-dim-resolvent}
If $G\in \ctq,$ then the solution $F\in \ctq$ is given by
\begin{eqnarray*}
\lefteqn{F(t,s)} \\
&=& G(t,s) \\
&&+\gamma \int_0^s\exp((\beta+\gamma)(s-y))G(t,y)dy \\
&&+\delta \int_0^t\exp((\alpha+\delta)(t-x))G(x,s)dx \\
&&+\gamma \dint \dt \modBessel{(t-x)}{(s-y)}\eadbg{(t-x)}{(s-y)}G(x,y)dydx \\
&&+\delta \dint \ds \modBessel{(t-x)}{(s-y)}\eadbg{(t-x)}{(s-y)}G(x,y)dydx \\
&&+2\gamma\delta \dint \modBessel{(t-x)}{(s-y)}\eadbg{(t-x)}{(s-y)}G(x,y)dydx.
\end{eqnarray*}
\end{lem}

\begin{proof}
We start with mulitiplying equation (\ref{2-dim-equation}) with $\exp(-\alpha t-\beta s).$ Thus,
\begin{eqnarray} \label{multiplied-2dim-eq}
\lefteqn{\exp(-\alpha t-\beta s)F(t,s)-} \\ \nonumber 
&& -\delta\int_0^t\exp(-\alpha \tau-\beta s)F(\tau,s)d\tau-\gamma\int_0^s\exp(-\alpha t-\beta \sigma)F(t,\sigma)d\sigma \\ \nonumber
&=&\exp(-\alpha t-\beta s)G(t,s).
\end{eqnarray}
Next we define
$$
H(t,s)=\dint \exp(-\alpha \tau-\beta \sigma)F(\tau,\sigma)d\sigma d\tau.
$$
As boundary conditions we have
$$
H(t,0)=H(0,s)=0 \ \forall \ 0\le s,t\le T.
$$
From (\ref{multiplied-2dim-eq}) and the definition of $H,$ we obtain the following differential equation:
\begin{equation} \label{H-diff-equation}
\dt\ds H(t,s)-\delta \ds H(t,s)-\gamma\dt H(t,s)=\exp(-\alpha t-\beta s)G(t,s).
\end{equation}

Next we write
$$
H(t,s)=\exp(\delta t +\gamma s)K(t,s).
$$
Then $K$ is a solution of the differential equation

\begin{equation} \label{K-diff-equation}
\dt\ds K(t,s)-\gamma\delta K(t,s)=\exp(-(\alpha+\delta)t-(\beta+\gamma)s)G(t,s).
\end{equation}
Using the transformation given in \cite[p.69,(6),p. 208,Nr.7]{VoelkerDoetsch},
and $K(t,0)=K(0,s)=0,$ we have
\begin{eqnarray*}
K(x,y)&=&\modBessel{x}{y}**\exp(-(\alpha+\delta)x-(\beta+\gamma)y)G(x,y)  \\
&=& \dint \modBessel{(t-x)}{(s-y)}\exp(-(\alpha+\delta)x-(\beta+\gamma)y)G(x,y)dydx.
\end{eqnarray*}
Transforming back gives
$$
\exp(-\alpha t -\beta s)F(t,s)=\dt\ds H(t,s)=\dt\ds\fk{\exp(\delta t +\gamma s)K(t,s)}.
$$
Consequently, we have to compute
\begin{eqnarray*}
\lefteqn{\ds\fk{\exp(\delta t +\gamma s)K(t,s)}} \\
&=& \gamma \edg \dint \modBessel{(t-x)}{(s-y)}\eadbgm{x}{y}G(x,y)dydx \\
&&+\edg \int_0^t\modBessel{(t-x)}{0}\eadbgm{x}{s}G(x,s)dx \\
&&+\edg \dint \ds \modBessel{(t-x)}{(s-y)}\eadbgm{x}{y}G(x,y)dydx.
\end{eqnarray*}
Applying $\modBessel{0}{}=1,$ leads to
\begin{eqnarray*}
\lefteqn{\ds\fk{\exp(\delta t +\gamma s)K(t,s)}} \\
&=& \gamma \edg \dint \modBessel{(t-x)}{(s-y)}\eadbgm{x}{y}G(x,y)dydx \\
&&+\edg \int_0^t\eadbgm{x}{s}G(x,s)dx \\
&&+\edg \dint \ds \modBessel{(t-x)}{(s-y)}\eadbgm{x}{y}G(x,y)dydx.
\end{eqnarray*}
Finally,
\begin{eqnarray*}
\lefteqn{\dt\ds\fk{\exp(\delta t +\gamma s)K(t,s)}} \\
&=& \gamma \delta \edg \dint \modBessel{(t-x)}{(s-y)}\eadbgm{x}{y}G(x,y)dydx \\
&&+\gamma \edg\int_0^s\eadbgm{t}{y}G(t,y)dx \\
&&+\gamma\edg\dint\dt \modBessel{(t-x)}{(s-y)}\eadbgm{x}{y}G(x,y)dydx \\
&&+\delta \edg\int_0^t\eadbgm{x}{s}G(x,s)dydx \\
&&+\edg\eadbgm{t}{s}G(t,s) \\
&&+\delta \edg \dint \ds \modBessel{(t-x)}{(s-y)}\eadbgm{x}{y}G(x,y)dydx \\
&&+\edg\int_0^s\ds \modBessel{0}{(s-y)}\eadbgm{x}{y}G(x,y)dydx \\
&&+\edg \dint \dt\ds \modBessel{(t-x)}{(s-y)}\eadbgm{x}{y}G(x,y)dydx .
\end{eqnarray*}
Using the known values,
\begin{eqnarray*}
\partial_2 \bk{(x_1,x_2)\mapsto \modBessel{(t-x_1)}{(s-x_2)}}(t,\alpha)&=&0, \\
\dt\ds \modBessel{(t-x)}{(s-y)}&=&\gamma \delta \modBessel{(t-x)}{(s-y)},
\end{eqnarray*}
it remains to multiply with $\exp(\alpha t +\beta s),$ and we are finished.

Similar to the initial value case we consider the whole line one, i.e., for $\om <0,$ we have to look for a positive resolvent in a
Banach lattice. We will compute
$(I-T_{\la,\mu})^{-1}$ for,
\begin{eqnarray*}
T_{\la, \mu}f(t,s) &:=&
\alm\intO \exp(-\dl \tau)f(t-\tau,s) d\tau \\
&&+\glm \intO \exp(-\dm \tau)f(t,s-\tau)d\tau\
\end{eqnarray*}
on $BUC(\re \times \re).$ 
\end{proof}

The next lemma provides a representation for the positive resolvent of the
operator $T_{\la,\mu}.$
We cite the result \cite[Lemma 4.2 p. 648]{Kreulichbd}.

\begin{lem} \label{main-ineq}
For the operator $T_{\la,\mu}$ defined above we have:
\begin{enumerate}
\item $\nrm{T_{\la,\mu}} \le \dlpmp{\la+\mu} < 1.$
Consequently,
$$\nrm{(I-T_{\lambda,\mu})^{-1}} \le \frac{\lpmplmw}{\lambda\mu\om},$$
and $ (I-T_{\lambda,\mu})^{-1} $ is positive.
\item 
Letting $\Lambda=\la+\mu-\la\mu\om,$ 
the resolvent is given by
\begin{eqnarray*}
\lefteqn{(I-T_{\lambda, \mu})^{-1}u(t,s) = u(t,s) + } \\
&&+\ \frac{\la}{\mu\Lambda} \intO\exp(\frac{\la\om-1}{\Lambda}y)u(t,s-y)dy \\
&&+\ \frac{\mu}{\la\Lambda} \intO\exp(\frac{\mu\om-1}{\Lambda}x)u(t-x,s)dx \\
&&+\ \frac{\la}{\mu\Lambda} \intO\intO \partial_xI_0\fk{2\sqrt{\frac{xy}{\Lambda^2}}}
            \exp\fk{\frac{\mu\om-1}{\Lambda}x+\frac{\la\om-1}{\Lambda}y}
      u(t-x,s-y)dydx \\
&&+\ \frac{\mu}{\la\Lambda} \intO\intO \partial_yI_0\fk{2\sqrt{\frac{xy}{\Lambda^2}}}
            \exp\fk{\frac{\mu\om-1}{\Lambda}x+\frac{\la\om-1}{\Lambda}y}
      u(t-x,s-y)dydx \\\\
&&+\ \frac{2}{\Lambda^2} \intO\intO I_0\fk{2\sqrt{\frac{xy}{\Lambda^2}}}
            \exp\fk{\frac{\mu\om-1}{\Lambda}x+\frac{\la\om-1}{\Lambda}y}
      u(t-x,s-y)dydx. \\ \\
\end{eqnarray*}
\end{enumerate}
\end{lem}

\begin{proof} The previous lemma could be derived from Lemma  \ref{general-2-dim-resolvent}, with the setting:
\begin{gather*}
"t",\quad \alpha:=-\frac{1}{\la},\quad \delta:=\frac{\mu}{\la(\lmw)},\mbox{ consequently } \alpha+\delta=\frac{\mu\om -1}{\lmw},\\
"s",\quad \beta:=-\frac{1}{\mu},\quad \gamma:=\frac{\la}{\mu(\lmw)},\mbox{ consequently } \beta+\gamma=\frac{\la\om -1}{\lmw}.
\end{gather*}
For fixed and small $\la,\mu>0,$ and $I_a=(a,\infty),$ ($a=-\infty$ possible), we consider the approximating operators,
$$
\Funk{T^a}{BUC(I_a\times I_a)}{BUC(I_a\times I_a)}{f}
{\frac{\mu}{ \la\Lambda}\int_a^t \exp(-\dl(t- \tau))f(\tau,s) d\tau +\frac{\la}{\mu\Lambda} \int_a^s \exp(-\dm (t-\tau))f(t,\tau)d\tau.}
$$
Further, use the help function
$$
H^a(t,s)=\int_a^t\int_a^s\exp(\frac{1}{\la}\tau+\frac{1}{\mu}\sigma)F(\tau,\sigma)d\sigma d\tau,
$$
with the boundary condition,
$$
H^a(t,a)=H^a(a,s)=0, \forall a\le t,s\le T.
$$
Letting
$$
K^a(x,y)=\int_a^t\int_a^s \modBesselL{(t-x)}{(s-y)}\exp(-\mlmw x-\llmw y)G(x,y)dydx,
$$ 
we find that $K^a$ satisfies the partial differential equation  (\ref{K-diff-equation}), with the boundary condition $K(t,a)=K(a,s)=0.$ Simple integration of the differential equation and the fact that $\bk{f\to\int_a^t\int_a^sf(\tau,\sigma)d\sigma d\tau}$ is quasinilpotent proves the uniqueness of the solution in $BUC(I_a\times I_a)$.
Thus, we are in the situation to follow the proof of Lemma \ref{general-2-dim-resolvent}, while differentiating a convolution integral, only the differentiation variable comes into play. 
\end{proof}

\end{document}